\documentclass[a4paper,11pt]{article}
\usepackage[utf8]{inputenc}
\usepackage[margin=1in]{geometry}
\pdfoutput=1

\providecommand{\keywords}[1]
{\small\textbf{Keywords ---} #1}

\usepackage[
     backend=biber,
     style=numeric,
     citestyle=numeric,
     natbib=true,
     sorting=nty,
     maxbibnames=99,
     doi=true,
]{biblatex}
\usepackage[english]{babel}
\usepackage[T2A]{fontenc}
\usepackage{csquotes}
\usepackage{graphicx}
\usepackage{color}
\usepackage{booktabs}
\usepackage{lipsum}
\usepackage{caption}
\usepackage{subcaption}
\usepackage{enumitem}
\usepackage{multirow}
\usepackage{float}
\usepackage[toc,page]{appendix}
\usepackage{xcolor}
\usepackage[automake,nonumberlist,nopostdot]{glossaries}
\usepackage{blindtext}
\usepackage{authblk}
\usepackage[ruled,vlined]{algorithm2e}
\usepackage{placeins}
\usepackage{amsmath}
\usepackage{amsfonts}
\usepackage{amssymb}
\usepackage{amsthm}
\usepackage{bm}
\usepackage{bbm}
\usepackage{verbatim}
\usepackage{algorithmic}
\usepackage{todonotes}
\usepackage[colorlinks=true,urlcolor=blue,linkcolor=blue,citecolor=blue]{hyperref}
\usepackage{authblk}
\usepackage{cancel}

\theoremstyle{plain}
\newtheorem{theorem}{Theorem}[section] 
\newtheorem{lemma}{Lemma}[section]
\newtheorem{corollary}{Corollary}[section]

\newtheorem{assumption}{Assumption}[]

\usepackage{mathtools}
\def\sqrtexplained#1{%
	\begingroup
	\sbox0{$#1$}
	\def\underbrace##1_##2{##1}
	\sbox2{$#1$}
	\dimen0=\wd0 \advance\dimen0-\wd2
	\mathrlap{\sqrt{\phantom{\displaystyle#1}\kern\dimen0 }}
	\hphantom{\sqrt{\vphantom{\displaystyle#1}}}
	\endgroup
	#1}
\date{} 
\DeclarePairedDelimiter{\chevrons}{\langle}{\rangle}

\DeclareMathOperator*{\linop}{\mathfrak{L}}

\DeclareMathOperator*{\expectation}{\mathbb{E}}
\DeclareMathOperator{\expectationc}{\mathbb{E}_j}
\newcommand{\expec}[1]{\expectation\left[#1\right]}
\newcommand{\expecc}[1]{\expectationc\left[#1\right]}
\newcommand{\mse}[1]{\textrm{MSE}\left(#1\right)}
\newcommand{\cost}[1]{\text{Cost}\left(#1\right)}
\newcommand{\abs}[1]{\left|#1\right|}
\newcommand{\norm}[1]{\left\|#1\right\|}
\DeclareMathOperator{\dif}{d\!} 
\DeclareMathOperator{\lebesgue}{L}

\DeclareMathOperator*{\argmin}{arg\,min}

\newcommand{\risk}{\ensuremath{\mathcal{R}}}

\newcommand{\obj}{\ensuremath{\mathcal{J}}}
\newcommand{\objest}{\hat{\obj}^{j}}
\newcommand{\objt}{\widetilde{\obj}^{j}}
\newcommand{\objgrad}{\nabla \obj}
\newcommand{\objestgrad}{\tilde{\nabla} \objest}\newcommand{\lagrangian}{\ensuremath{\mathcal{L}}}
 
\newcommand{\qoi}{\ensuremath{Q}}
\newcommand{\linf}{\mathrm{L}^{\infty}}
\newcommand{\lp}{\mathrm{L}^{p}}
\newcommand{\ellq}{\mathrm{L}^{q}}

\newcommand{\sltwo}{l^{2}}

\newcommand{\measure}{\ensuremath{\mathbb{P}}}

\newcommand{\setR}{\ensuremath{\mathbb{R}}}
\newcommand{\setN}{\ensuremath{\mathbb{N}}}
\makeatletter
\newcommand{\abbrie}{i.e\@ifnextchar.{}{.\@}}
\newcommand{\abbreg}{e.g\@ifnextchar.{}{.\@}}
\newcommand{\abbrae}{a.e\@ifnextchar.{}{.\@}}
\newcommand{\abbrst}{s.t\@ifnextchar.{}{.\@}}
\makeatother

\newcommand{\interpone}[1]{\mathcal{S}_n' \left(#1\right)}

 \newcommand{\interp}[1]{\mathcal{S}_n \left(#1\right)}

\newcommand{\thetab}{\bm\theta}

\newcommand{\Phieone}{\hat{\Phi}_L'}
\newcommand{\Phietwo}{\hat{\Phi}_L''}

\newcommand{\Psieone}{\hat{\Psi}_L'}

\newcommand{\Psiekone}{\hat{\Psi}_{L,k}'}

\newcommand{\Psik}{\Psi_k}
\newcommand{\Phie}{\hat{\Phi}_L}
\newcommand{\Psie}{\hat{\Psi}_L}
\newcommand{\Psiek}{\hat{\Psi}_{L,k}}

\newcommand{\Phione}{\Phi'}

\newcommand{\Psione}{\Psi'}

\newcommand{\Psikone}{\Psi_k'}

\newcommand{\quant}{q_{\tau}}

\newcommand{\cvar}{c_{\tau}}

\newcommand{\err}{\text{e}}
\newcommand{\indicator}{\mathbbm{1}}
\newcommand{\errest}{\hat{\err}}

\usepackage{stmaryrd}
\DeclarePairedDelimiter{\sqbbracket}{\llbracket}{\rrbracket}
\newcommand{\Zint}[2]{\sqbbracket{#1,#2}}
\newcommand{\tavg}[2][{}]{\chevrons{#2}_{#1}}
\newcommand{\deriv}[3][{}]{\frac{\dif^{#1}#2}{\dif{}#3^{#1}}}

\newcommand{\inner}[1]{\left\langle #1 \right\rangle}
\newtheorem{proposition}{Proposition}[section]

\SetKw{And}{and}
\SetKw{Or}{or}
\SetKw{Then}{then}
\SetKw{Elif}{elif}
\SetKw{Not}{not}
\SetKw{Update}{update}
\SetKw{Compute}{compute}
\SetKw{Estimate}{estimate}
\SetKwComment{Mycomment}{}{}
\SetCommentSty{textit}
\SetKwProg{Fn}{}{}{}
\SetKwSwitch{Switch}{Case}{Other}{switch}{do}{case}{otherwise}{endcase}{endsw}

\newacronym{qoi}{QoI}{Quantity of Interest}
\newacronym{pde}{PDE}{Partial Differential Equation}
\newacronym{ode}{ODE}{Ordinary Differential Equation}
\newacronym{kde}{KDE}{Kernel Density Estimation}
\newacronym{mlmc}{MLMC}{Multi-Level Monte Carlo}
\newacronym{mc}{MC}{Monte Carlo}
\newacronym{cmlmc}{CMLMC}{Continuation MLMC}
\newacronym{mimc}{MIMC}{Multi-Index Monte Carlo}
\newacronym{mfmc}{MFMC}{Multi-Fidelity Monte Carlo}
\newacronym{mse}{MSE}{Mean Squared Error}
\newacronym{cdf}{CDF}{Cumulative Distribution Function}
\newacronym{pdf}{PDF}{Probability Density Function}
\newacronym{cvar}{CVaR}{Conditional-Value-at-Risk}
\newacronym{var}{VaR}{Value-at-Risk}
\newacronym{ouu}{OUU}{Optimization Under Uncertainty}
\newacronym{uq}{UQ}{Uncertainty Quantification}
\newacronym{sde}{SDE}{Stochastic Differential Equation}
\newacronym{rhs}{RHS}{Right Hand Side}
\newacronym{lhs}{LHS}{Left Hand Side}
\newacronym{amgd}{AMGD}{Alternating Minimisation-Gradient Descent}
\makeglossaries

\title{Gradient-based optimisation of the conditional-value-at-risk using the multi-level Monte Carlo method}

\author[1]{Sundar Ganesh\thanks{sundar.ganesh@epfl.ch}}
\author[1]{Fabio Nobile\thanks{fabio.nobile@epfl.ch, corresponding author}}
\affil[1]{\small{Institute of Mathematics, {\'E}cole Polytechnique F{\'e}d{\'e}rale de Lausanne, 1015 Lausanne, Switzerland}}

\begin{document}

\maketitle
\vspace{-0.25in}

\begin{abstract}
In this work, we tackle the problem of minimising the Conditional-Value-at-Risk (CVaR) of output quantities of complex differential models with random input data, using gradient-based approaches in combination with the Multi-Level Monte Carlo (MLMC) method.
In particular, we consider the framework of multi-level Monte Carlo for parametric expectations introduced in \cite{krumscheid2017multilevel} and propose modifications of the MLMC estimator, error estimation procedure, and adaptive MLMC parameter selection to ensure the estimation of the CVaR and sensitivities for a given design with a prescribed accuracy.
We then propose combining the MLMC framework with an alternating inexact minimisation-gradient descent algorithm, for which we prove Q-linear convergence in the optimisation iterations under the assumptions of strong convexity and Lipschitz continuity of the gradient of the objective function.
We demonstrate the performance of our approach on two numerical examples of practical relevance, which evidence the same optimal asymptotic cost-tolerance behaviour as standard MLMC methods for fixed design computations of output expectations.
\end{abstract}

\noindent \keywords{Multilevel Monte Carlo Methods, VaR, CVaR, Uncertainty Quantification, Optimisation Under Uncertainty, Gradient Descent.}

\section{Introduction}\label{sec:introduction}
Optimisation algorithms play an important role across various scientific and engineering fields as valuable design tools.
The key goal of optimisation is to find the best values of certain parameters (design variables) of a model, typically a differential model such as a \gls{pde}, used to predict the behaviour of a certain system, such that a desired output \gls{qoi} of the model is optimised.
Such differential models usually also include various other input parameters besides the design variable, which may or may not be fully characterised.
There is an increasing interest in the computational science and engineering community to treat such parameters as random variables to reflect their uncertainty, either due to a lack of knowledge or to some intrinsic variability.
As a result, the output \gls{qoi} being optimised also becomes a random variable. 
Naively optimising the system for only one particular value of the inputs (e.g., the nominal value) can lead to a design that is not robust enough to the uncertainties in the system.
A classical example is of civil engineering structures designed to minimise structural loads for moderate wind conditions, which are then unable to withstand local wind gusts or storms \cite{kodakkal2022risk}.

The field of \gls{pde}-constrained \gls{ouu} seeks to characterise the randomness of the output \gls{qoi} of the \gls{pde} using summary statistics such as moments, quantiles, etc., and optimise the summary statistic instead of the \gls{qoi} directly.
In particular, in risk-averse \gls{pde}-constrained optimisation, one aims at favouring designs with acceptable performance also in extreme conditions.
In this case, the summary statistic, often called a risk-measure, should quantify the importance that is given in the design process to unfavourable scenarios.
The reader is referred to \cite{shapiro2009lectures} for a comprehensive review of several risk-measures and their main properties.
An important class of risk-measures is that of coherent risk-measures \cite{ruszczynski2006optimization,artzner1999coherent}, which exhibit favourable properties such as monotonicity and convexity.

In this work, we focus on the so-called \gls{cvar} \cite{rockafellar2002conditional}, which corresponds to the expectation of the output \gls{qoi} conditional on being above a given quantile (referred to as the \gls{var} in the finance literature), and is a widely used coherent risk-measure.
To be more specific, let $\qoi(z) \in \setR$ denote the random \gls{qoi}, which depends on the design parameter $z \in \setR^d$.
We denote by $\mathbb{E}[X]$ the expected value of a random variable $X$, and by $\cvar(z)$ the \gls{cvar} of $\qoi(z)$ of significance $\tau \in [0,1]$, i.e., $\cvar(z) \coloneqq \expec{\qoi(z) | \qoi(z) \geq \quant}$, where $\quant$ is the $\tau$-quantile of $\qoi(z)$.
It was demonstrated in \cite{rockafellar2002conditional} that $\cvar(z)$ could be written in the following form under certain conditions on the distribution of $\qoi(z)$:
\begin{align}
\cvar(z) = \min_{\theta \in \setR} \big\{\Phi(\theta;z) \coloneqq \expec{\phi(\theta, \qoi(z))}\big\} ,\quad \phi(\theta, \qoi ) \coloneqq \theta + \frac{(\qoi-\theta)^+}{1-\tau},\label{eq:cvar_pexp_def}
\end{align}
where we denote with $X^+$ the positive part of $X$; namely $X^+ \coloneqq \max(0,X)$.
In this work, we consider the following problem of penalised \gls{cvar} minimisation:
\begin{align}
\obj^* &= \min_{z \in \setR^d}\left\{ \cvar(z) + \kappa \norm{z-z_{ref}}_{\sltwo}^2 \right\},\label{eq:opt_form_orig}\\
&= \min_{\substack{z \in \setR^d \\ \theta \in \setR}} \bigg\{ \obj(\theta,z) \coloneqq \Phi(\theta;z) + \kappa \norm{z-z_{ref}}_{\sltwo}^2 \bigg\},\label{eq:opt_form_combined_spec}
\end{align}
where we have added a term penalising deviation of the design $z$ from a preferred design $z_{ref}$.
In Eq.~\eqref{eq:opt_form_combined_spec}, the parameter $\kappa$ controls the strength of the penalisation, and $\norm{\cdot}_{\sltwo}$ denotes the Euclidean norm.
In particular, we will consider Monte Carlo type approximations of problem~\eqref{eq:opt_form_combined_spec}, and since evaluating the objective function and its sensitivities at a given design $z$ requires the solution of a costly \gls{pde} many times, we will accelerate the Monte Carlo estimation by multilevel strategies following the well established \gls{mlmc} paradigm \cite{giles2015multilevel}, which has been shown to provide significant performance improvements in comparison to classical Monte Carlo methods for estimating various summary statistics of output \gls{qoi} of differential models \cite{giles2008multilevel,giles2015multilevel,hoel2012adaptive,Pisaroni_mlmcPart1_pre,collier2014continuation}.

Two broad approaches can be used to solve problem~\eqref{eq:opt_form_combined_spec}; namely, gradient-free and gradient-based methods.
Evolutionary algorithms, a type of gradient-free method, were used in combination with Monte Carlo estimators for \gls{pde}-constrained \gls{cvar} minimisation in \cite{quagliarella2020risk, quagliarella2019value}.
A genetic algorithm was also used in combination with \gls{mlmc} estimators in \cite{pisaroni2019continuation}.
Multiple different risk-measures, including the \gls{cvar}, were estimated, and the framework was applied to aerodynamic shape optimisation problems.
The authors of \cite{chaudhuri2020multifidelity} proposed a multifidelity Monte Carlo estimator for the \gls{cvar} based on cross-entropy methods combined with importance sampling.
This approach was used in \cite{peherstorfer2018multifidelity} in combination with gradient-free optimisation algorithms to minimise the \gls{cvar}. 
However, gradient-free algorithms typically have slower rates of convergence in comparison to gradient-based methods and involve multiple expensive evaluations of the objective function.
We propose instead the use of gradient based algorithms, combined with \gls{mlmc} estimators, to compute sensitivities of the objective function in problem~\eqref{eq:opt_form_combined_spec}.
In particular, the \gls{mlmc} estimators developed in this work rely on the framework of parametric expectations \cite{krumscheid2017multilevel} and extend the work in \cite{Ganesh2022a} to the computation of \gls{cvar} sensitivities, addressing the corresponding challenges as outlined hereafter.
We mention that gradient-based methods have also been used for other risk-measures and sampling strategies in the context of \gls{pde}-constrained \gls{ouu} \cite{guth2022parabolic,van2019robust,martin2019multilevel}. 

The computation of the sensitivities of the \gls{cvar} $\cvar(z)$ with respect to the extended design variables $z$ and $\theta$ typically requires the estimation of expectations of the form $\expec{(\qoi(z)-\theta)^+}$ and $\expec{\indicator_{\qoi(z) \geq \theta}f(z)}$ for suitable design-dependent random variables $f(z)$.
Although $\expec{(\qoi(z)-\theta)^+}$ and $\expec{\indicator_{\qoi(z) \geq \theta}f(z)}$ can be shown to be differentiable in $\theta$ and $z$ \cite{hong2009simulating, hong2011monte} under some conditions on the distribution of $\qoi(z)$, sample- or quadrature-based approximations of these expectations are typically not differentiable and may require some additional treatment.
One possibility is to directly use the non-differentiable estimations in combination with non-smooth optimisation techniques that use sub-gradient information.
For example, the work in \cite{lim2010portfolio} uses a combination of smooth and non-smooth optimisation techniques, using sub-gradients computed using Monte Carlo estimators, to minimise the \gls{cvar}.
Alternatively, one could construct smoothed versions of the maximum/indicator functions, with sufficient regularity such that sample-based approximations are still differentiable.
For example, a regularised version of the \gls{cvar} was constructed in \cite{kouri2016risk}, with second order differentiability, and optimised successfully using a trust-region method.
However, although regularised or smoothed versions of the \gls{cvar} can be constructed with adequate differentiability, this property is lost in the limit of vanishing smoothing, as is required when the algorithm is close to the optimum.
The method proposed in \cite{krumscheid2017multilevel,Ganesh2022a} offers an alternative to \gls{cvar} regularisation.
In these works, the quantity $\expec{(\qoi(z)-\theta)^+}$ is estimated directly using an \gls{mlmc} estimator at a set of points in $\theta$, all sharing the same realisations of $\qoi(z)$, followed by a cubic spline interpolation over the pointwise evaluations thus obtained.
Derivatives such as $\expec{\indicator_{\qoi(z)\geq\theta}f(z)}$ are then approximated using derivatives of the cubic spline.
We propose to follow the above path in this work.
As was discussed in \cite{Ganesh2022a}, directly using a naive \gls{mlmc} estimator to estimate $\expec{\indicator_{\qoi(z)\geq\theta}f(z)}$ causes non-optimal \gls{mlmc} complexity behaviour.
By constructing an \gls{mlmc} estimator of $\expec{(\qoi(z)-\theta)^+f(z)}$ and numerically differentiating in $\theta$ instead, the approach in \cite{Ganesh2022a} ameliorates this issue and preserves the same optimal complexity behaviour of the \gls{mlmc} method as predicted for estimating $\expec{\qoi(z)}$.
Lastly, since the \gls{mlmc} estimator proposed in \cite{krumscheid2017multilevel,Ganesh2022a} automatically provides an approximation $\hat{\obj}(\cdot,z)$ of the function $\theta \mapsto \obj(\theta,z)$ at a given design $z$, we propose in this work to use an optimisation algorithm in which, at each iteration, gradient steps are taken only in the design variable $z$, whereas exact optimisation in $\theta$ is performed using the surrogate $\hat{\obj}(\cdot,z)$.
Such an algorithm, introduced in \cite{exaqute_d6.3}, was applied in combination with the Monte Carlo estimation of a regularised version of the \gls{cvar} in \cite{beiser2020adaptive}.

The main contributions of this work are as follows.
We propose novel expressions for the sensitivity of the objective function defined in Eq.~\eqref{eq:opt_form_combined_spec} in terms of parametric expectations, thus allowing us to use and extend the framework in \cite{Ganesh2022a} to build cost optimal adaptive \gls{mlmc} estimators for those sensitivities with error control.
We then propose to use \gls{mlmc} sensitivity estimators within an \gls{amgd} algorithm, analogous to the one proposed in \cite{exaqute_d6.3,beiser2020adaptive}, where gradient steps are taken in the design variable $z$ whereas exact optimisation is performed in $\theta$ using an \gls{mlmc}-constructed surrogate $\hat{\obj}$ of $\obj$.
The accuracy of the surrogate and sensitivity estimation is increased over the optimisation iterations and is set proportional to the gradient norm.
Following closely the analysis in \cite{beiser2020adaptive}, we propose a convergence result for our algorithm under the assumption that the objective function $\obj(\theta,z)$ is strongly convex with Lipschitz continuous gradients.

The structure of this paper is as follows.
We present the problem formulation in Section~\ref{sec:prob_form}, for a problem of penalised \gls{cvar} minimisation of the form in Eq.~\eqref{eq:opt_form_combined_spec}.
The novel expression for the gradients in terms of the parametric expectations is also presented in Section~\ref{sec:prob_form}.
In Section~\ref{sec:conv_novel_alg}, we propose the \gls{amgd} algorithm with inexact gradient and objective function estimation and demonstrate its convergence.
Section~\ref{sec:grad_est_mlmc} discusses the novel \gls{mlmc} estimators, error estimation procedure, and adaptive \gls{cmlmc}-type hierarchy selection for the gradients of $\obj(\theta,z)$.
In addition, it presents a final \gls{cmlmc}-\gls{amgd} algorithm.
Lastly, in Section~\ref{sec:results}, we demonstrate the above optimisation algorithm and \gls{mlmc} procedure on two problems of interest.
The first is a two-dimensional oscillator, typically used to model oscillatory phenomena in excitable media. 
The second is a more applied problem of pollutant transport modelling. 
We demonstrate that the procedure proposed in this work performs well and reflects the theoretical results presented in Sections~\ref{sec:prob_form} and~\ref{sec:conv_novel_alg}.

\section{Problem formulation}\label{sec:prob_form}
Let $(\Omega, \mathcal{F}, \measure)$ denote a complete probability space, $\omega \in \Omega$ an elementary random event and $z \in \setR^d$ the vector of design variables.
We denote by $\qoi(z,\omega) \in \setR$ the random \gls{qoi}, typically a functional of the solution to an underlying differential model with random input $\omega$ and design $z$.
We are interested in minimising the \gls{cvar} $\cvar(z)$ of the random variable $\qoi(z,\cdot)$ over the designs $z \in \setR^d$, as indicated in Eq.~\eqref{eq:opt_form_orig}, following the formulation presented in \cite{rockafellar2002conditional}. 
To this end, we first introduce the following assumptions on the random variable $\qoi(z,\cdot)$.
\begin{assumption}\label{assume:qoi_prop}
For any $z \in \setR^d$:
\begin{enumerate}[label=(\roman*)]
\item $\qoi(z,\cdot)$ is a random variable in $\lp(\Omega, \setR)$ for some $p \in [1,\infty)$.
\item The measure of $\qoi(z,\cdot)$ admits a probability density function, i.e., the measure of $\qoi(z,\cdot)$ is free of atoms.
We denote by $\Gamma$ the subset of random variables in $\lp(\Omega, \setR)$ that are free of atoms, and hence, $\qoi(z,\cdot) \in \Gamma \subset \lp(\Omega, \setR)$. \label{assume:qoi_prop:point:atom_free}
\item There exists a positive random variable $K$, possibly dependent on $z$, such that $\expec{K} < \infty$ and
\begin{align}
\abs{\qoi(z+\Delta z, \cdot)-\qoi(z, \cdot)} \leq K(\cdot) \norm{\Delta z}_{\sltwo},
\end{align}
for any $\Delta z \in \setR^d$ close enough to $0$ (restated here from \cite{hong2009simulating, hong2011monte}).\label{assume:qoi_prop:point:diff_bound}
\item For almost every $\omega \in \Omega$, the mapping $z \mapsto \qoi(z,\omega)$ is differentiable in $\setR^d$ and the corresponding vector of partial derivatives $\qoi_z(z,\cdot) = \left[\qoi_{z^1}(z,\cdot),...,\qoi_{z^d}(z,\cdot)\right]^T$ of $\qoi$ with respect to the components $z^k$ of $z,\; k \in \{1,...d\}$, is a random variable in $\lp(\Omega, \setR^d)$.\label{assume:qoi_prop:point:qoi_sens}
\end{enumerate}
\end{assumption}

To quantify the tails of $\qoi(z,\cdot)$, we first define the $\tau$-\gls{var} $\quant(z)$, alternatively known as the $\tau$-quantile, of significance $\tau \in (0,1)$ as follows:
\begin{align}
\quant(z) \coloneqq \min \{\theta \in \setR | \expec{\indicator_{\qoi(z,\cdot)\leq \theta}} \geq \tau \}.
\end{align}
The $\tau$-\gls{cvar} $\cvar(z)$ is defined as the expected value of $\qoi(z,\cdot)$ in the tail above and including the $\tau$-\gls{var} $\quant(z)$:
\begin{align}
\cvar(z) \coloneqq \expec{\qoi(z,\cdot) | \qoi(z,\cdot) \geq \quant(z)}.
\end{align}
As was described in Section~\ref{sec:introduction}, \cite{rockafellar2002conditional} proposed that $\cvar(z)$ could be written in the form in Eq.~\eqref{eq:cvar_pexp_def} for a random variable $\qoi(z,\cdot)$ satisfying Assumption~\ref{assume:qoi_prop}.\ref{assume:qoi_prop:point:atom_free}.

In this work, we extensively use the concept of parametric expectations.
In particular, let us introduce the function (parametric expectation) $\Phi:\setR \times \setR^d \to \setR$ as:
\begin{align}
\Phi(\theta;z) \coloneqq \expec{\phi(\theta,\qoi(z, \cdot))}, \quad \theta \in \setR,\; z \in \setR^d, \label{eq:parest_form}
\end{align}
with $\phi:\setR \times \setR \to \setR$ given by:
\begin{align}
\phi(\theta, \qoi) &\coloneqq \theta + \frac{(\qoi-\theta)^+}{1-\tau}, \quad\theta \in \setR,\;\qoi \in \setR. \label{eq:phi_form}
\end{align}
The introduction of the parametric expectation $\Phi$ has the advantage that the $\tau$-\gls{var} $\quant(z)$ and the $\tau$-\gls{cvar} $\cvar(z)$ of any significance $\tau$ can be obtained by simple post-processing of $\Phi$ as:
\begin{align}
\quant(z) = \argmin_{\theta \in \setR}\Phi(\theta;z), \qquad \cvar(z) &= \min_{\theta \in \setR}\Phi(\theta;z) = \Phi(\quant(z);z).\label{eq:stat_defs}
\end{align}
The framework of parametric expectations allows us to write the penalised \gls{cvar} minimisation problem in Eq.~\eqref{eq:opt_form_orig} as a combined minimisation over $\theta$ and $z$ as in Eq.~\eqref{eq:opt_form_combined_spec}.
The problem is restated below for reference:
\begin{align}
\obj^* &= \min_{\substack{z \in \setR^d \\ \theta \in \setR}} \bigg\{ \obj(\theta,z) \coloneqq \Phi(\theta;z) + \kappa \norm{z-z_{ref}}_{\sltwo}^2 \bigg\}. \label{eq:opt_form_combined}
\end{align}
For the remainder of this work, we address the challenge of solving problem~\eqref{eq:opt_form_combined}.
The combined objective function $\obj(\theta,z)$ has several properties that, when combined with the properties of $\qoi$ in Assumption~\ref{assume:qoi_prop}, have useful implications for gradient based optimisation techniques.
We first discuss the differentiability of $\obj(\theta, z)$.
Theorem~\ref{thm:q_t_diff} below gives a result on Fr\'echet differentiability of the \gls{cvar}.
\begin{theorem}\label{thm:q_t_diff}
Let $\linop(X,Y)$ denote the space of bounded linear operators between the normed vector spaces $X$ and $Y$. 
We define the function $\risk: \setR \times \lp(\Omega;\setR) \to \setR$ as follows:
\begin{align}
\risk(\theta, \qoi) \coloneqq \theta + \frac{\expec{(\qoi-\theta)^+}}{1-\tau} = \expec{\phi(\theta,\qoi)}.
\end{align}
Then, $\risk(\theta,\qoi)$ is jointly Fr\'echet differentiable in $\setR  \times  \Gamma$, with Fr\'echet derivative $D\risk(\theta, \qoi) \in \linop(\setR \times \lp(\Omega, \setR), \setR)$ at the point $(\theta, \qoi) \in \setR \times \Gamma$ in the direction $(\delta \theta, \delta \qoi) \in \setR \times \lp(\Omega, \setR)$ given by:
\begin{align}
D \risk (\theta, \qoi) (\delta \theta, \delta \qoi) = \left(1 - \frac{\expec{\indicator_{\qoi \geq \theta}}}{1-\tau}\right)\delta \theta + \frac{\expec{\indicator_{\{\qoi \geq \theta\}}\delta \qoi}}{1-\tau}. \label{eq:J_frechet}
\end{align}
\end{theorem}
\begin{proof}
The reader is referred to Appendix~\ref{sec:proof-q-t-diff} for the proof.
\end{proof}
\noindent This result, combined with Assumption~\ref{assume:qoi_prop} on $\qoi$, leads immediately to the differentiability of $\obj(\theta,z)$.
\begin{corollary}\label{corollary:j_frech}
The objective function $\obj(\theta, z)$ is jointly Fr\'echet differentiable in $\setR \times \setR^d$, with Fr\'echet derivative $D\obj(\theta,z) \in \linop(\setR \times \setR^d, \setR)$ at the point $(\theta,z)$ in the direction $(\delta \theta, \delta z) \in \setR \times \setR^d$ given by:
\begin{align}
D \obj (\theta, z) (\delta \theta, \delta z) = \left(1 - \frac{\expec{\indicator_{\qoi \geq \theta}}}{1-\tau}\right)\delta \theta + \frac{\expec{\indicator_{\{\qoi \geq \theta\}} \qoi_z^T \delta z}}{1-\tau} + 2 \kappa(z-z_{ref})^T \delta z.
\end{align}
\end{corollary}

A direct implication of Corollary~\ref{corollary:j_frech} is that the gradient $\objgrad \in \setR^{d+1}$ and the partial derivatives $\obj_z(\theta, z) = \left[\obj_{z^1}(\theta,z),...,\obj_{z^d}(\theta,z)\right]^T$ and $\obj_\theta(\theta, z)$ exist and are given by the following expressions:
\begin{align}
\obj_{\theta}(\theta,z) &= 1 - \frac{\expec{\mathbbm{1}_{\qoi(z,\cdot) \geq \theta}}}{1-\tau},\label{eq:sens_theta}\\
\obj_{z}(\theta,z) &= \frac{\expec{\mathbbm{1}_{\qoi(z,\cdot) \geq \theta}\qoi_{z}(z,\cdot)}}{1-\tau} + 2 \kappa(z-z_{ref}). \label{eq:sens_z}
\end{align}
One of the main contributions of this work is the estimation of the sensitivities in Eqs.~\eqref{eq:sens_theta} and \eqref{eq:sens_z} using \gls{mlmc} estimators.
However, as discussed in Section~\ref{sec:introduction}, using \gls{mlmc} to directly estimate the expectations in Eqs.~\eqref{eq:sens_theta} and \eqref{eq:sens_z} may result in compromised or non-optimal \gls{mlmc} performance.
The reader is referred to \cite{Ganesh2022a,krumscheid2017multilevel} for a detailed discussion on the topic.
To ameliorate this issue, we propose the following alternative formulation of the gradients in terms of parametric expectations:
\begin{align}
\obj_{\theta}(\theta,z) &= \Phione(\theta;z) \label{eq:par_theta}, \quad \text{with $\Phi$ as in Eqs.~\eqref{eq:parest_form}-\eqref{eq:phi_form}},\\
\obj_{z}(\theta,z) &= \Psione(\theta;z) + 2 \kappa(z-z_{ref}), \label{eq:par_sens_theta}\\
\text{where } \Psi(\theta;z) &\coloneqq \expec{ - \frac{(\qoi(z,\cdot)-\theta)^+\qoi_{z}(z,\cdot)}{1-\tau}} \eqqcolon \left[\expec{\psi(\theta, \qoi, \qoi_{z^1})},...,\expec{\psi(\theta, \qoi, \qoi_{z^d})}\right]^T.\label{eq:psi_def}
\end{align}
The superscript prime of the parametric expectations in Eq.~\eqref{eq:par_theta} and Eq.~\eqref{eq:par_sens_theta} denotes the derivative computed with respect to $\theta$.
In addition to $\Phi(\theta;z)$, we have introduced the parametric expectation $\Psi(\theta;z) \in \setR^d$ and the function $\psi(\theta, \qoi, \qoi_{z^k}) \in \setR$ where $z^k$ and $\qoi_{z^k}$ denote the $k^{\text{th}}$ components of $z$ and $\qoi_z$ respectively, $k \in \{1,...,d\}$.
The differentiability of $\Psi(\theta; z)$ in $\theta$ follows by the same arguments of Theorem~\ref{thm:q_t_diff} and Corollary~\ref{corollary:j_frech}, under Assumption~\ref{assume:qoi_prop}.
It was shown in \cite{Ganesh2022a} that since $\phi$ and $\psi$ are Lipschitz continuous in their arguments, the corresponding \gls{mlmc} estimators no longer suffer from the compromised performance due to discontinuities.
The idea is then to build \gls{mlmc} estimators $\hat{\Phi}(\cdot,z)$ and $\hat{\Psi}(\cdot,z)$ for the whole functions $\theta \mapsto \Phi(\theta; z)$ and $\theta \mapsto \Psi(\theta;z)$ respectively on a suitably chosen interval $\Theta \subset \setR$, and then approximate $\obj_{\theta}$ and $\obj_z$ as $\hat{\obj}_{\theta}(\theta,z) = \hat{\Phi}'(\theta;z)$ and $\widetilde{\obj}_{z}(\theta,z) = \hat{\Psi}'(\theta;z) + 2\kappa(z-z_{ref})$ respectively.
As a by-product of this approach for estimating sensitivities, we construct an approximation $\theta \in \Theta \mapsto \hat{\obj}(\theta,z) = \hat{\Phi}(\theta;z) + \kappa \norm{z-z_{ref}}_{\sltwo}^2$ of the objective function itself for all $\theta \in \Theta$, at a given design $z \in \setR^d$.
This allows us to consider an optimisation problem in which exact minimisation in $\theta$ is performed at each iteration using the surrogate $\hat{\obj}$, and gradient steps are performed only in $z$ using the approximate gradient $\widetilde{\obj}_z$.
Notice that the gradient approximation in $z$ is inconsistent with the surrogate model $\hat{\obj}$, i.e., $\widetilde{\obj}_z \neq \partial_z \hat{\obj}$, in contrast to $\hat{\obj}_{\theta}$.
We will detail this approach in the next section.

\section{Gradient based optimisation algorithm} \label{sec:conv_novel_alg}
In this section, we present a gradient-based iterative procedure to find a local minimiser $(\theta^*, z^*)$ of the \gls{ouu} problem in Eq.~\eqref{eq:opt_form_combined}, should it exist. 
The broad goal of a gradient based algorithm is to define the iterates $(\theta_j, z_j), j \in \setN$ such that
\begin{align}
\lim_{j \to \infty} (\theta_j,z_j) = (\theta^*,z^*), \label{eq:convergence}
\end{align}
where the iterates are computed using gradient information.
Motivated by our interest in using \gls{mlmc} estimators based on parametric expectations to estimate the objective function and its sensitivities, we consider in this section the general situation in which, at each iteration $j$ of the gradient based algorithm, we build an approximation $\objest(\theta,z), \theta \in \Theta$ of the objective function at the design $z\in \setR^d$ on a suitably chosen interval $\Theta \subset \setR$ (which may depend on $j$, although to ease the notation, we do not highlight such dependence), as well as approximations $\objest_{\theta}(\theta,z)$ and $\objt_z(\theta,z)$, $\theta \in \Theta$, where the approximation $\objt_z$ may not coincide with the $z$-derivative of $\objest$.
The approximations $\objest$, $\objest_\theta$ and $\objt_z$ may be random, as will be the case for \gls{mlmc} estimators.
We then propose the following variation of the standard gradient descent algorithm, starting from an initial design $z_0$:
\begin{align}
\theta_j \in \argmin_{\theta \in \Theta} \objest(\theta,z_j),\label{eq:step_theta}\\
z_{j+1} =  z_j - \alpha \objt_z(\theta_j,z_j) , \label{eq:step_z}
\end{align}
where $\alpha > 0$ denotes a step size parameter.
We note that according to the procedure in \cite{Ganesh2022a}, the interval $\Theta$ can be freely selected and, hence, we can ensure that $\theta_j$ always belongs to the interior of $\Theta$, so that $\objest_\theta(\theta_j, z_j) = 0 \;\forall j \in \setN$.

In Theorem~\ref{thm:approx_convergence} in Section~\ref{sec:exp_conv}, we show that the iterates $(\theta_j, z_j)$ converge Q-linearly in the iteration counter $j$ towards $(\theta^*,z^*)$ under additional assumptions on the objective function $\obj$ and its approximations $\objest$.
The results of Theorem~\ref{thm:approx_convergence}, specifically the implications of Eq.~\eqref{eq:gradient_error_bound_assumption} introduced there, demonstrate that Q-linear convergence of the iterates $z_j$ and $\theta_j$ in $j$ can be obtained if the gradient approximation is accurate up to a tolerance that is a fraction $\eta$ of the gradient magnitude at the previous iteration.
Such an accuracy condition was used in \cite{beiser2020adaptive}, and has been utilised in literature prior to this work.
The interested reader is referred to \cite{byrd2012sample} and \cite{bollapragada2018adaptive}.

The step size is selected sufficiently small, and remains fixed over all optimisation iterations, although variable step sizes and line search methods could be easily added.
The algorithm is terminated once the gradient magnitude drops to a specified fraction of the initial value.
We introduce here the notation $w = (\theta, z)$, $\objgrad = (\obj_{\theta}, \obj_z)$ and $\objestgrad = (\objest_{\theta}, \objt_z)$ for convenience in the following.\\ \ \\
\begin{algorithm}[H]
\begin{algorithmic}[1]
	\STATE Input: Initial design $z_0$, iteration counter $j=0$, tolerance $0<\epsilon <1$, step size $\alpha > 0$ and tolerance fraction $\eta > 0$.
	\STATE Set residual $r = \epsilon+1$.
	\WHILE{$r > \epsilon$}
		\STATE \textbf{if} $j = 0$ $\{$ Compute $\hat{\obj}^{0}(\cdot, z_j)$ and $\widetilde{\obj}_z^{0}(\cdot, z_j)$ up to a fixed tolerance.$\}$ 
		\STATE \textbf{else} $\bigg\{$ Compute $\objest(\cdot,z_j)$ and $\objt_z(\cdot,z_j)$ such that $\mse{\objestgrad(\cdot,z_j)} \leq \eta^2 \norm{\objgrad(\theta_{j-1},z_j)}^2_{\sltwo}$ with $\mse{\objestgrad(\cdot,z_j)}$ defined as in Eq.~\eqref{eq:gradient_error_bound_assumption}.$\bigg\}$
		\STATE Compute a minimiser $\theta_{j} \in \argmin_{\theta \in \Theta} \objest(\theta, z_j)$.
		\STATE Compute gradient step $z_{j+1} = z_j - \alpha \objt_{z}(\theta_j, z_j) $.
		\STATE Set residual $r = \norm{\objestgrad(w_{j})}_{\sltwo}^2 / \norm{\tilde{\nabla} \hat{\obj}^{0}(w_{0})}_{\sltwo}^2$.
		\STATE Update $j \leftarrow j + 1$.
	\ENDWHILE
\end{algorithmic}
\caption{Novel \gls{amgd} algorithm}
\label{alg:opti_mlmc_demo}
\end{algorithm}

\subsection{Convergence analysis} \label{sec:exp_conv}
For the interested reader, we present a self-contained convergence analysis of the iterates $(\theta_j, z_j)$ in Theorem~\ref{thm:approx_convergence}, under additional assumptions on $\obj$ and $\objest$, based on the analysis presented in \cite{beiser2020adaptive}.
The key differences in the two analyses are related to the fact that the algorithm studied here is an \gls{amgd} algorithm instead of a pure gradient descent algorithm.
We first note that the objective function $\obj(\theta,z)$ is convex under the additional assumption that $\qoi(z,\cdot)$ is almost surely convex in $z$ \cite[Theorem 10]{rockafellar2002conditional}.
When combined with the assumption that $\obj \to \infty$ when $\norm{z}_{\sltwo}, \abs{\theta} \to \infty$, this ensures that a minimiser of $\obj(\theta,z)$ exists in $\setR \times \setR^d$.
However, we require additional assumptions on the objective function $\obj$ to prove Q-linear convergence of the iterates $\theta_j$ and $z_j$ towards such a minimiser; namely Assumptions~\ref{assume:strong_convexity} and \ref{assume:lipschitz_grad} below on strong convexity and on the Lipschitz continuity of the gradients, respectively.
An immediate implication of Assumption~\ref{assume:strong_convexity} is that there exists a unique minimiser $(\theta^*, z^*) \in \setR \times \setR^d$ for the \gls{ouu} problem in Eq.~\eqref{eq:opt_form_combined} such that $\obj_{z}(\theta^*, z^*) =  \obj_{\theta}(\theta^*, z^*) = 0$.

In what follows, we denote by $\expecc{\cdot}$ the expectation conditional on all of the random variables used to define $z_j$ (i.e., conditioned on the past up to iteration $j$), and by $\inner{\cdot, \cdot}$ the $\sltwo$ inner product.
Readers interested in the implementation details of Algorithm~\ref{alg:opti_mlmc_demo} and its relation to the \gls{mlmc} method can proceed directly to Section~\ref{sec:grad_est_mlmc}.
\begin{assumption}\label{assume:strong_convexity}
The objective function $\obj$ is $\mu$-strongly convex, i.e., there exists $\mu > 0$ such that, for all $w_a, w_b \in \setR \times \setR^d$, equivalently:
\begin{enumerate}
\item[(i)] $\obj(w_b) \geq \obj(w_a) + \inner{w_b-w_a,\objgrad(w_a)} + \frac{\mu}{2} \norm{w_b-w_a}_{\sltwo}^2$,
\item[(ii)] $\inner{\objgrad(w_b) - \objgrad(w_a),w_b-w_a} \geq \mu \norm{w_b-w_a}_{\sltwo}^2$.
\end{enumerate}
\end{assumption}
\begin{assumption}\label{assume:lipschitz_grad}
The objective function $\obj$ has Lipschitz continuous gradients, i.e., there exists $L > 0$ such that, for all $w_a, w_b \in \setR \times \setR^d$:
\begin{align}
\norm{\objgrad(w_b)-\objgrad(w_a)}_{\sltwo} \leq L \norm{w_b-w_a}_{\sltwo}.
\end{align}
\end{assumption}
\begin{lemma} \label{lemma:21}
Let $\obj$ satisfy Assumptions~\ref{assume:strong_convexity} and Assumptions~\ref{assume:lipschitz_grad}.
Then we have that, for $0 < \alpha \leq 1/L$,
\begin{align}
\frac{\mu}{2} \norm{w-w^*}_{\sltwo}^2 + \frac{\alpha}{2} \norm{\objgrad(w)}^2_{\sltwo} \leq \inner{\objgrad(w),w-w^*}.
\end{align}
The above result is restated here from \cite[Lemma 2.1]{beiser2020adaptive}.
\end{lemma}

\begin{theorem} \label{thm:approx_convergence}
Let $\Theta \subset \setR$ be a convex set.
Let $\obj:\setR \times \setR^d \to \setR$ satisfy Assumptions~\ref{assume:strong_convexity} and \ref{assume:lipschitz_grad}, and $\objest:\Theta \times \setR^d \to \setR$ satisfy the following condition:
\begin{align}
\mse{\objestgrad(\cdot,z_j)} &\coloneqq \expecc{\norm{\objest_{\theta}(\cdot, z_j)-\obj_{\theta}(\cdot, z_j)}^2_{\linf(\Theta)}} \nonumber \\
&+ \sum_{k=1}^d \expecc{\norm{\objt_{z,k}(\cdot, z_j)-\obj_{z,k}(\cdot, z_j)}^2_{\linf(\Theta)}} \leq \eta^2 \norm{\objgrad(\theta_{j-1}, z_j)}^2_{\sltwo},\label{eq:gradient_error_bound_assumption}
\end{align}
for some $\eta > 0$, where $(\theta_{j-1},z_j)$ is the $j^{\text{th}}$ iterate produced by Algorithm~\ref{alg:opti_mlmc_demo} with step size $\alpha$ satisfying $0<\alpha \leq 1/L$ and $\alpha \mu \leq 1$.
Then, the following result holds true:
\begin{align}
\expec{\norm{z_{j+1}-z^*}_{\sltwo}^2 + C_1 (\theta_{j}-\theta^*)^2} & \leq \xi \expec{ \norm{z_j-z^*}_{\sltwo}^2 + C_1 (\theta_{j-1}-\theta^*)^2 }, \label{eq:z_contraction_approx}
\end{align}
for some constants $C_1 > 0$ and $0 <\xi < 1$.
\end{theorem}
\begin{proof}
From the definition of the iterate $z_{j+1}$ in Eq.~\eqref{eq:step_z}, we have:
\begin{align}
\norm{z_{j+1} - z^*}_{\sltwo}^2 &= \norm{z_j - z^* - \alpha  \objt_z(\theta_j, z_j) }_{\sltwo}^2\\
&= \norm{z_j -z^*}_{\sltwo}^2 + \alpha^2 \norm{\objt_{z}(\theta_j, z_j)}^2 - 2\alpha \inner{ \objt_{z}(\theta_j, z_j),z_j - z^*}\\
&= \norm{z_j -z^*}_{\sltwo}^2 + \underbrace{\alpha^2\left( \norm{\objt_{z}(\theta_j, z_j)}^2 + \left(\objest_{\theta}(\theta_j, z_j)\right)^2\right)}_{\eqqcolon \hat{T}_1}\nonumber\\
&\underbrace{- 2\alpha \left( \inner{ \obj_{z}(\theta_j, z_j),z_j - z^*} + \inner{ \obj_{\theta}(\theta_j, z_j),\theta_j - \theta^*} \right)}_{\eqqcolon \hat{T}_2} \nonumber  \\
&\underbrace{- 2\alpha \left( \inner{ \objt_{z}(\theta_j, z_j)-\obj_{z}(\theta_j, z_j),z_j - z^*} + \inner{\objest_{\theta}(\theta_j, z_j)-\obj_{\theta}(\theta_j, z_j),\theta_j - \theta^*}\right)}_{\eqqcolon \hat{T}_3}.
\end{align}
The term $\hat{T}_1 = \alpha^2 \norm{\objestgrad(w_j)}^2_{\sltwo}$ can be bounded as follows:
\begin{align}
\expecc{\hat{T}_1} &= \alpha^2 \expecc{ \norm{\objestgrad(\theta_j, z_j)}_{\sltwo}^2} \\
&\leq \alpha^2 \expecc{ \norm{\objestgrad(\theta_j, z_j) \pm \objgrad(\theta_j, z_j)}_{\sltwo}^2} \\
&\leq \alpha^2 \left[ \expecc{ \norm{\objestgrad(\theta_j, z_j) - \objgrad(\theta_j, z_j)}_{\sltwo}^2}^{1/2} + \expecc{\norm{\objgrad(\theta_j,z_j)}_{\sltwo}^2}^{1/2}\right]^2 \\
&\leq \alpha^2 \left[ \eta \norm{\objgrad(\theta_{j-1}, z_j)}_{\sltwo} + \expecc{\norm{\objgrad(\theta_j,z_j)}_{\sltwo}^2}^{1/2}\right]^2 \\
&\leq \alpha^2 \left[ (\eta^2 + \eta) \norm{\objgrad(\theta_{j-1}, z_j)}_{\sltwo}^2 + (1+\eta)\expecc{\norm{\objgrad(\theta_j,z_j)}_{\sltwo}^2}\right],
\end{align}
The term $\hat{T}_2 = -2\alpha \inner{\objgrad(w_j),w_j-w^*}$ can be bounded as follows:
\begin{align}
\expecc{\hat{T}_2} &\leq -\alpha \mu \left(\norm{z_j-z^*}_{\sltwo}^2 + \expecc{(\theta_j-\theta^*)^2}\right) - \alpha^2 \expecc{\norm{\objgrad(\theta_j,z_j)}_{\sltwo}^2},
\end{align}
where we have used Lemma~\ref{lemma:21}.
Finally, the term $\hat{T}_3 = -2\alpha \inner{\objestgrad(w_j)-\objgrad(w_j),w_j-w^*}$ can be bounded as follows:
\begin{align}
\expecc{\hat{T}_3} &\leq 2\alpha \expecc{\norm{\objestgrad(\theta_j, z_j)-\objgrad(\theta_j, z_j) }_{\sltwo}  \norm{w_j - w^*}_{\sltwo}} \\
&\leq 2\alpha \expecc{\norm{\objestgrad(\theta_j, z_j)-\objgrad(\theta_j, z_j) }^2_{\sltwo}}^{1/2}  \expecc{\norm{w_j - w^*}_{\sltwo}^2}^{1/2} \\
&\leq 2\alpha \eta \norm{\objgrad(\theta_{j-1}, z_j)}_{\sltwo} \expecc{\norm{w_j - w^*}^2_{\sltwo}}^{1/2}.
\end{align}
Combining the bounds for $\hat{T}_1$, $\hat{T}_2$ and $\hat{T}_3$, we have the following:
\begin{align}
\expecc{\norm{z_{j+1}-z^*}_{\sltwo}^2} &\leq (1-\alpha\mu) \norm{z_j-z^*}_{\sltwo}^2 -  \alpha \mu \expecc{(\theta_j-\theta^*)^2} \nonumber \\
&+ \alpha^2 (\eta^2 + \eta) \norm{\objgrad(\theta_{j-1}, z_j)}_{\sltwo}^2 + \alpha^2 \eta\expecc{\norm{\objgrad(\theta_j,z_j)}_{\sltwo}^2} \nonumber\\
&+ 2\alpha \eta \norm{\objgrad(\theta_{j-1}, z_j)}_{\sltwo} \expecc{\norm{w_j - w^*}^2_{\sltwo}}^{1/2} . \label{eq:sum_terms_bound_expecc}
\end{align}
We now utilise Lemma~\ref{lemma:21} once again, from which we have the following result:
\begin{align}
\alpha \norm{\objgrad(w)}_{\sltwo} \leq (1+\sqrt{1-\alpha \mu}) \norm{w-w^*}_{\sltwo} \eqqcolon \tilde{L} \norm{w-w^*}_{\sltwo}, \label{eq:lemma_21_corollary}
\end{align}
for $0 < \alpha \leq 1/L$ and $\alpha \mu \leq 1$.
In addition, the last term of Eq.~\eqref{eq:sum_terms_bound_expecc} can be rewritten as follows:
\begin{align}
2\alpha \eta \norm{\objgrad(\theta_{j-1}, z_j)}_{\sltwo} \expecc{\norm{w_j - w^*}^2_{\sltwo}}^{1/2} \leq \eta \left(\frac{\alpha^2 \norm{\objgrad(\theta_{j-1}, z_j)}^2_{\sltwo}}{\tilde{L}} + \tilde{L} \expecc{\norm{w_j-w^*}_{\sltwo}^2}\right) \label{eq:t3_new_bound}
\end{align}
Applying Eqs.~\eqref{eq:lemma_21_corollary} and \eqref{eq:t3_new_bound} to Eq.~\eqref{eq:sum_terms_bound_expecc}, we then have the following simplified bound:
\begin{align}
\expecc{\norm{z_{j+1}-z^*}_{\sltwo}^2} &\leq \left(1-\alpha \mu + (\eta^2 + 2 \eta) \tilde{L}^2 + 2 \eta \tilde{L} \right) \norm{z_j-z^*}_{\sltwo}^2 \nonumber \\
& + \left(-\alpha \mu + \eta \tilde{L}^2 + \eta \tilde{L} \right) \expecc{(\theta_j-\theta^*)^2} \nonumber \\
& + \left((\eta^2 + \eta)\tilde{L}^2+\eta \tilde{L}\right) (\theta_{j-1}-\theta^*)^2,\\
&= (1-C_1+C_2) \norm{z_j-z^*}_{\sltwo}^2 - C_1 \expecc{(\theta_j-\theta^*)^2} + C_2(\theta_{j-1}-\theta^*)^2,
\end{align}
where we have defined the constants $C_1 = \alpha \mu - \eta \tilde{L}^2 + \eta \tilde{L}$ and $C_2 = (\eta^2 + \eta)\tilde{L}^2 + \eta \tilde{L}$. 
We then have the following:
\begin{align}
\expecc{\norm{z_{j+1}-z^*}_{\sltwo}^2} + C_1 \expecc{(\theta_j-\theta^*)^2} &\leq (1-C_1+C_2) \norm{z_j-z^*}_{\sltwo}^2 + C_2 (\theta_{j-1}-\theta^*)^2 \nonumber\\
&\leq \max\left( 1-C_1+C_2, \frac{C_2}{C_1} \right) \left( \norm{z_j-z^*}_{\sltwo}^2 + C_1 (\theta_{j-1}-\theta^*)^2 \right).
\end{align}
We note that the leading constant on the right hand side is less than $1$ as long as $C_1>C_2$, which holds true for $\eta < \sqrt{1+\alpha\mu/\tilde{L}^2} - 1$.
This in turn ensures contraction in the norm $\norm{z}_{\sltwo}^2 + C_1 \theta^2$ on the space $\setR^d \times \setR$.
This completes the proof.
\end{proof}
\paragraph{Remark 1.} We note that although the accuracy condition Eq.~\eqref{eq:gradient_error_bound_assumption} is stated in the $\linf$-norm for all $\theta$, the proof of Theorem~\ref{thm:approx_convergence} uses this property only at $\theta_j$.
This condition is required since we do not know the quantile $\theta_j$ a priori, and seek to use the parametric expectation framework from \cite{Ganesh2022a} to do so.
\cite{Ganesh2022a} requires that the error in the approximations $\objest$ be controlled at all $\theta$, in order to estimate $\theta_j$ accurately.

\paragraph{Remark 2.} In practical applications, it is difficult to determine whether Assumptions~\ref{assume:strong_convexity} and~\ref{assume:lipschitz_grad} are satisfied, since both are strongly dependent on the properties of the random \gls{qoi} $\qoi(z,\cdot)$.
These assumptions require stronger properties on $\qoi(z,\cdot)$ and its \gls{pdf} than those presented in Assumption~\ref{assume:qoi_prop}; for example, that the \gls{pdf} remains both upper bounded and lower bounded away from zero for all designs $z$, and that the random variable $\qoi(z,\cdot)$ is bounded, i.e., $\qoi(z,\cdot) \in \linf(\Omega,\setR)$.

\paragraph{Remark 3.} We remark that $\eta$ is a monotonically increasing function of the product $\alpha \mu$ in the interval $0 < \alpha \mu \leq 1$.
For an appropriate step size $\alpha$, chosen such that $\alpha \mu$ is close to $1$, $\eta$ can be as large as $0.4$.

\section{Gradient estimation and error control using \gls{mlmc} methods} \label{sec:grad_est_mlmc}
We note that the key assumption in the proof of Theorem~\ref{thm:approx_convergence} is Eq.~\eqref{eq:gradient_error_bound_assumption}; namely, that the gradient approximation is accurate up to a tolerance that is proportional to the magnitude of the true gradient.
As stated earlier in Section~\ref{sec:introduction}, we are interested in utilising the framework of \gls{mlmc} estimators for parametric expectations developed in \cite{Ganesh2022a} for the accurate estimation of the objective function $\obj$ (risk-measure \gls{cvar}) and its gradient.

Expressing the gradients $\obj_z$ and $\obj_{\theta}$ in terms of the first derivatives of the parametric expectations $\Phi(\theta;z)$ and $\Psi(\theta;z)$ as in Eqs.~\eqref{eq:par_theta} and \eqref{eq:par_sens_theta} and estimating the latter using \gls{mlmc} estimators poses many key advantages.
The first advantage was already seen earlier in Section~\ref{sec:conv_novel_alg}; namely that $\objt_z$ and $\objest_\theta$ can be estimated for all $\theta$ for a given design $z$ in one shot.
Secondly, as was demonstrated in \cite{Ganesh2022a}, the level-wise differences for the \gls{mlmc} estimator of $\Phi(\theta;z)$, analogous to the level-wise differences corresponding to the classical \gls{mlmc} estimator of $\expec{\qoi}$, decay at the same rate in the levels $l$ as the differences $\qoi_l - \qoi_{l-1}$, in an appropriately selected norm over $\theta \in \setR$.
This ensures that if cost-optimal \gls{mlmc} behaviour can be achieved for estimating $\expec{\qoi}$, then it can be achieved also for \gls{mlmc} estimators of $\Phi(\theta;z)$ and $\Psi(\theta;z)$, using a practically computable number of samples.
The last key advantage is that, using the mechanism in \cite{Ganesh2022a}, one can select the parameters of the \gls{mlmc} estimator such that a prescribed tolerance can be attained on the \gls{mlmc} approximation error on $\Phi$ and $\Psi$.
By prescribing a tolerance proportional to the gradient magnitude, one can estimate the gradient using \gls{mlmc} estimators that respect the condition in Eq.~\eqref{eq:gradient_error_bound_assumption} as required by Algorithm~\ref{alg:opti_mlmc_demo}.

Although the procedure used in this work to estimate $\Phi$ accurately is identical to the one described in \cite{Ganesh2022a}, some important modifications are required to use the same procedure for accurately estimating $\Psi$.
We present in this section the modifications of the work developed in \cite{Ganesh2022a} that are required for the accurate estimation of $\Psi$, and consequently the gradients $\obj_\theta$ and $\obj_z$, using the \gls{mlmc} method.

\subsection{\gls{mlmc} estimator for the gradients}
We begin by recalling that the parametric expectation $\Psi$ is defined as in Eq.~\eqref{eq:psi_def}.
The proposed \gls{mlmc} method relies on a sequence of approximations $\{\qoi_{l}(z) \}_{l=0}^L$ to $\qoi(z)$ on a sequence of $L+1$ discretisations with, for example, different mesh sizes $h_0 > h_1 > ... > h_L$, typically a geometric sequence $h_{l-1}=s h_l$ with $s>1$.
The \gls{mlmc} estimator for the $k^{\text{th}}$ component $\Psik(\cdot;z) \coloneqq \expec{\psi(\cdot, \qoi(z), \qoi_{z^k}(z))}$ of $\Psi$ on $\Theta$, $k \in \{1,...,d\}$ follows the same construction as that for $\Phi$ in \cite{Ganesh2022a}.
The first step is to estimate $\Psik(\theta_r,z), r \in \{1,...,n\}$, on a set of $n$ equidistant points $\thetab = \{\theta_1, ..., \theta_n\}$ such that $\Theta = [\theta_1, \theta_n]$, by a standard \gls{mlmc} estimator $\Psiek(\theta_r;z)$, which reads:
\begin{align}
\Psiek(\theta_r;z) &\coloneqq \frac{1}{N_0}\sum_{i=1}^{N_0} \psi\left(\theta_r, \qoi_{0}^{(i,0)}(z),\qoi_{z^k,0}^{(i,0)}(z)\right) \nonumber\\
&+\sum_{l=1}^L \frac{1}{N_l} \sum_{i=1}^{N_l} \left[\psi\left(\theta_r, \qoi_{l}^{(i,l)}(z), \qoi_{z^k,l}^{(i,l)}(z)\right) - \psi\left(\theta_r, \qoi_{{l-1}}^{(i,l)}(z), \qoi_{{z^k,l-1}}^{(i,l)}(z)\right) \right], \label{eq:mlmc_est_1}
\end{align}
where $\qoi_{l}^{(i,l)}(z) \equiv \qoi_l(z;\omega^{(i,l)})$ and $\qoi_{{l-1}}^{(i,l)}(z) \equiv \qoi_{l-1}(z;\omega^{(i,l)})$ are correlated realisations of $\qoi_l(z)$ and $\qoi_{l-1}(z)$, respectively, typically obtained by solving the underlying differential problem on meshes with discretisation parameters $h_l$ and $h_{l-1}$, driven by the same realisation $\omega^{(i,l)}$ of the random parameters for the fixed design $z$. 
On the other hand, $\qoi_l^{(i,l)}$ and $\qoi_k^{(j,k)}$ are independent if $i \neq j$ or $l \neq k$.
Finally, $\qoi_{z^k,l}^{(i,l)}$ and $\qoi_{z^k,l-1}^{(i,l)}$ are the sensitivities of the realisations $\qoi_{l}^{(i,l)}$ and $\qoi_{l-1}^{(i,l)}$ respectively with respect to $z^k$.
$\{N_l\}_{l=0}^{L}$ is a decreasing sequence of sample sizes. 
The \gls{mlmc} hierarchy is hence defined by three parameters; namely the number of interpolation points $n$, the number of levels $L$ and the level-wise sample sizes $N_l$.

We finally construct a \gls{mlmc} estimator $\Psiek$ of the whole function $\Psik(\cdot;z):\Theta \to \setR$ by interpolating over the pointwise estimates as below:
\begin{align}
\Psiek(\cdot;z) = \interp{\Psiek(\thetab;z)},
\end{align}
where $\mathcal{S}_n$ denotes a uniform cubic spline interpolation operator and $\Psiek(\thetab;z)$ denotes the set of pointwise \gls{mlmc} estimates in Eq.~\eqref{eq:mlmc_est_1}, that is $\Psiek(\thetab;z) = \{\Psiek(\theta_1;z), \Psiek(\theta_2;z), \dots, \Psiek(\theta_n;z)\}$.
An estimate of the first derivative $\Psikone$ in $\theta$ is then obtained by computing the derivative of the resultant interpolated function, for each component $\Psiekone$:
\begin{align}
\Psiekone(\cdot;z) &\coloneqq \interpone{ \Psiek(\thetab;z) } \coloneqq \frac{\partial }{\partial \theta } \interp{ \Psiek(\thetab;z) }.\label{eq:mlmc_est_2}
\end{align}

\subsection{Estimation of the \gls{mse} of the gradient}
Since we have assumed that the gradient estimate $\objestgrad$ is a random vector in $\lp(\Omega,\setR^{d+1})$ with $p \geq 2$, we propose to quantify the error on the gradient in an \gls{mse} sense as follows:
\begin{align}
\mse{\objestgrad(\cdot, z_j)} \coloneqq \expec{\norm{\objest_{\theta}(\cdot, z_j) - \obj_{\theta}(\cdot, z_j) }_{\linf(\Theta)}^2} + \sum_{k=1}^d \expec{\norm{\objt_{z,k}(\cdot, z_j) - \obj_{z,k}(\cdot, z_j) }_{\linf(\Theta)}^2},\label{eq:grad_error_def}
\end{align}
where $\obj_{z,k}$ and $\objt_{z,k}$ denote the $k^{\text{th}}$ components of $\obj_z$ and $\objt_z$.

We now present a result relating $\mse{\objestgrad(\cdot,z_j)}$ to the \gls{mse} of the \gls{mlmc} estimators $\Phieone$ and $\Psieone$.
\begin{proposition}\label{prop:gradient_error_bound}
Let $\Phie(\cdot;z_j)$ and $\Psie(\cdot;z_j)$ denote the \gls{mlmc} estimators of $\Phi(\cdot;z_j)$ and $\Psi(\cdot;z_j)$ as defined in \cite{Ganesh2022a} and Eq.~\eqref{eq:mlmc_est_1} respectively.
Let $\objestgrad(\cdot,z_j)$ be the approximation to the true gradient $\objgrad(\cdot,z_j)$ computed using the estimates $\Phieone(\cdot;z_j)$ and $\Psieone(\cdot;z_j)$ at the $j^{\text{th}}$ optimisation iteration.
Let $\Psik$ and $\Psiek$ denote the $k^{th}$ component of $\Psi$ and $\Psie$ respectively, for $k \in \{1,...,d\}$.
Let the \glsplural{mse} on $\Phieone$ and $\Psiekone$ be defined as follows:
\begin{align}
\mse{\Phieone}(z_j) &\coloneqq \expec{\norm{\Phieone(\cdot;z_j)-\Phione(\cdot;z_j)}_{\linf(\Theta)}^2},\\
\mse{\Psiekone}(z_j) &\coloneqq \expec{\norm{\Psiekone(\cdot;z_j)-\Psikone(\cdot;z_j)}_{\linf(\Theta)}^2}, \label{eq:mse_psi_thm}
\end{align}
for the design $z_j \in \setR^d$.
Then, we have that:
\begin{align}
\mse{\objestgrad(\cdot,z_j)} = \mse{\Phieone}(z_j) + \sum_{k=1}^d \mse{\Psiekone}(z_j).\label{eq:grad_error_bound}
\end{align}
\end{proposition}
\begin{proof}
We first note that:
\begin{align}
\norm{ \objest_{\theta}(\cdot,z_j)- \obj_{\theta}(\cdot,z_j)}_{\linf(\Theta)}^2 &=  \norm{\hat{\Phi}'(\cdot;z_j)-\Phione(\cdot;z_j)}_{\linf(\Theta)}^2\\
\norm{ \objt_{z,k}(\cdot,z_j) - \obj_{z,k}(\cdot,z_j)}_{\linf(\Theta)}^2 &=  \norm{ \Psiekone(\cdot;z_j)-\Psikone(\cdot;z_j) }_{\linf(\Theta)}^2 
\end{align}
Adding together each of the contribuitons and taking the expectation on both sides, we have that:
\begin{align}
\mse{\objestgrad(\cdot,z_j)} = \mse{\Phieone}(z_j) + \sum_{k=1}^d \mse{\Psiekone}(z_j). 
\end{align}
\end{proof}
As was described earlier in this section, we seek to use the error estimation and adaptivity procedure described in \cite{Ganesh2022a} to accurately estimate $\Phione$ and $\Psikone$, and consequently, to accurately estimate the gradient $\objgrad$.
From Eq.~\eqref{eq:grad_error_bound}, it is evident that if one can control the \gls{mse} of $\Phieone$ and $\Psiekone$ in an $\linf$ sense, one can control the \gls{mse} on the gradient $\objgrad$ as defined in Eq.~\eqref{eq:grad_error_def}.
Specifically, the \gls{mse} of the gradient is equal to a simple sum of the \glsplural{mse} of the parametric expectations.
Eq.~\eqref{eq:grad_error_bound} hence allows us to use the work of \cite{Ganesh2022a} to accurately calibrate \gls{mlmc} estimators for the parametric expectations $\Phione$ and $\Psione$ such that the resultant gradient estimate is accurate up to a prescribed tolerance.

\subsection{Modified error estimation procedure} \label{sec:error_mod}
Since the error estimation procedure is independent of the design $z$, in the following, we drop the explicit dependence of $\Phi$ and $\Psi$ on $z$, with the dependence being implied.
We recall here that the error estimation procedure for estimating $\mse{\Phieone}$ is identical to that presented in \cite{Ganesh2022a}.
The procedure for estimating $\mse{\Psieone}$ however has several modifications from the procedure for $\Phieone$, that we detail in this section.
We recall that $\mse{\Psieone}$ was defined in Eq.~\eqref{eq:mse_psi_thm}.
Proceeding similarly as in \cite{Ganesh2022a}, we can bound $\mse{\Psieone}$ as follows:
\begin{align}
\mse{\Psiekone} \leq (\errest^{\Psi_k}_{i})^2 + (\errest^{\Psi_k}_{b})^2 + (\errest^{\Psi_k}_{s})^2, \label{eq:err_three_split}
\end{align}
where $\errest^{\Psi_k}_{i}$, $\errest^{\Psi_k}_{b}$ and $\errest^{\Psi_k}_{s}$ denote error estimators that estimate the error due to interpolation, the error due to approximation of the \gls{qoi} (i.e. bias error), and the error due to finite sampling (i.e. statistical error) respectively on $\hat{\Psi}_{L,k}$.
The reader is referred to \cite{Ganesh2022a} for a detailed discussion of each of the three errors components, as well as their corresponding estimators. 

The procedure for estimating the interpolation and bias errors requires the accurate estimation of $\theta$-derivatives of the function $\Psi_{l,k}(\theta) = \expec{\psi(\theta,\qoi_l, \qoi_{z^k,l})}$.
Although the true function $\Psi_{l,k}$ is smooth, replacing the true probability density with an empirical probability density corresponding to a Monte Carlo estimator implies that the right-hand side would be a linear combination of piecewise linear functions.
The first derivative of such a function would be piecewise constant, and high order derivatives would not exist in the discontinuity points, and would be zero otherwise.
A \gls{mlmc} hierarchy designed based on estimates obtained in this manner would lead to non-optimal complexity behaviour.
In \cite[Section~3.2]{Ganesh2022a}, a \gls{kde} based procedure was described for ameliorating this issue.
Although the error estimation procedure is broadly the same for estimating $\Psi$ as for $\Phi$, an important distinction arises with respect to this \gls{kde} procedure, which we detail in this section.

Since the issue chiefly relates to the regularity of the empirical Monte Carlo probability density, we propose the use a \gls{kde} based smoothing technique; namely, we replace the true joint density $p_l$ of $(\qoi_l, \qoi_{z^k,l})$ with a \gls{kde} smoothed joint probability density $p_l^{kde}$, which consists of a linear combination of two-dimensional kernels composed of products of two one-dimensional Gaussian kernels centred on each of the $N_l$ fine samples $\{(\qoi^{(i,l)}_l, \qoi^{(i,l)}_{z^k,l})\}_{i=1}^{N_l}$:
\begin{align}
\Psi_{l,k}(\theta) &= \int \int \psi (\theta,q,q_{z^k}) p_l(q,q_{z^k}) dq dq_{z^k}\\
&\approx \int \int \psi(\theta, q , q_{z^k}) p^{kde}_l(q,q_{z^k}) dq dq_{z^k}\\
&\coloneqq \frac{1}{N_l}\sum_{i=1}^{N_l} \int \int \psi(\theta,q , q_{z^k}) K_{\delta_l}(q,\qoi_l^{(i,l)})K_{\delta_{z^k,l}}(q_{z^k},\qoi_{z^k,l}^{(i,l)}) dq dq_{z^k}\\
&= - \frac{1}{N_l}\sum_{i=1}^{N_l}\int q_{z^k} K_{\delta_{z^k,l}}(q_{z^k},\qoi_{z^k,l}^{(i,l)}) dq_{z^k} \int \frac{(q-\theta)^+}{1-\tau} K_{\delta_l}(q,\qoi_l^{(i,l)})dq\\
&= - \frac{1}{N_l}\sum_{i=1}^{N_l} \qoi_{z^k,l}^{(i,l)} \int \frac{(q-\theta)^+}{1-\tau} K_{\delta_l}(q,\qoi_l^{(i,l)})dq \eqqcolon \mathbb{E}^{kde}_{l,k}\left[ \psi(\theta,\cdot,\cdot)\right]. \label{eq:kde_def}
\end{align}
Here, $K_{\delta_l}(\cdot, \mu)$ denotes a Gaussian kernel with mean $\mu$ and bandwidth parameter $\delta_l >0$, which is selected according to Scott's rule \cite{scott1979optimal} for the realisations $\{\qoi_{l}^{(i,l)}\}_{i=1}^{N_l}$ and controls the ``width'' of the kernel.
A closed form expression can be computed for the integral in Eq.~\eqref{eq:kde_def}, leading to the \gls{kde} smoothened approximation $\mathbb{E}^{kde}_{l,k}\left[ \psi(\theta,\cdot,\cdot)\right]$ for $\Psik$.

According to the procedure in \cite{Ganesh2022a}, the interpolation error requires the estimation of the quantity $\norm{\Psik^{(4)}}$, for which we use the \gls{kde} estimator described above.
To this end, we first select a level $\lceil L/2 \rceil$ from the \gls{mlmc} hierarchy; this choice of level is to ensure that $\hat{\Psi}_{\lceil L/2 \rceil,k}$ is sufficiently close to $\Psik$, and $N_{\lceil L/2 \rceil}$ is large enough for the \gls{kde} procedure to produce accurate estimates. 
We then construct the \gls{kde} approximation $\Upsilon_{\lceil L/2\rceil,k}(\theta) \coloneqq \mathbb{E}^{kde}_{\lceil L/2\rceil, k}\left[ \psi(\theta,\cdot,\cdot)\right]$.
The fourth derivative $\Upsilon^{(4)}_k$ is then constructed using a second order central finite difference scheme on a uniform grid on $\Theta$ with $n' \gg n$ points.
The norm is evaluated on the same grid as follows:
\begin{align}
\norm{\Psik^{(4)}}_{\linf(\Theta)} \approx \max_{i \in \{1,...,n'\}} \left|\Upsilon^{(4)}_{\lfloor L/2 \rfloor,k} (\theta_i)\right|.
\end{align}
For the bias error on $\Psiek$, we are required to estimate the quantity
\begin{align}
\norm{\interpone{\expec{\psi(\theta,\qoi_l,\qoi_{z^k,l})-\psi(\theta,\qoi_{l-1},\qoi_{z^k,l-1})}}}_{\linf(\Theta)}. \label{eq:true_bias}
\end{align}
Replacing the expectation by a Monte Carlo estimator leads to the same regularity issue as described earlier in this section. 
To smooth the empirical Monte Carlo density, we propose the use of a \gls{kde} smoothed approximation $p_{l,l-1}^{kde}$ to the true density $p_{l,l-1}$ of $(\qoi_l, \qoi_{z^k,l},\qoi_{l-1}, \qoi_{z^k,l-1})$, consisting of products of four one-dimensional Gaussian kernels:
\begin{align}
&\expec{\psi(\theta,\qoi_l,\qoi_{z^k,l})-\psi(\theta,\qoi_{l-1},\qoi_{z^k,l-1})} \\
&= \int \int \int \int \left[ \psi(\theta,q^f ,q^f_{z^k})-\psi(\theta,q^c ,q_{z^k}^c)\right] p_{l,l-1}(q^f,q^f_{z^k},q^c,q^c_{z^k}) dq^f dq^f_{z^k} dq^c  dq_{z^k}^c\\ 
&\approx\frac{1}{N_l}\sum_{i=1}^{N_l} \int \int \int \int \left[ \psi(\theta,q^f ,q^f_{z^k})-\psi(\theta,q^c ,q_{z^k}^c)\right] \nonumber\\
&\times K_{\delta_l}(q^f,\qoi_{l}^{(i,l)})K_{\delta_{z^k,l}}(q^f_{z^k},\qoi_{z^k,l}^{(i,l)}) K_{\delta_{l-1}}(q^c,\qoi_{l-1}^{(i,l)}) K_{\delta_{z^k,l-1}}(q^c_{z^k},\qoi_{z^k,l-1}^{(i,l)}) dq^f dq^f_{z^k} dq^c  dq_{z^k}^c\\ 
&=\frac{1}{N_l}\sum_{i=1}^{N_l} \qoi_{z^k,l-1}^{(i,l)} \int  \frac{(q^c-\theta)^+}{1-\tau} K_{\delta_{l-1}}(q^c,\qoi_{l-1}^{(i,l)}) dq^c - \qoi_{z^k,l}^{(i,l)} \int \frac{(q^f-\theta)^+}{1-\tau} K_{\delta_l}(q^f,\qoi_{l}^{(i,l)}) dq^f\\
&\eqqcolon \mathbb{E}^{kde}_{l,l-1,k}\left[\psi(\theta,\qoi_l,\qoi_{z^k,l})-\psi(\theta,\qoi_{l-1},\qoi_{z^k,l-1}) \right]. \label{eq:kde_bivariate}
\end{align}
The expectation in Eq.~\eqref{eq:true_bias} can be replaced by the \gls{kde} smoothened expectation in Eq.~\eqref{eq:kde_bivariate}, which can then be used in the bias error estimation procedure outlined in \cite{Ganesh2022a}. 
Lastly, the procedure for the statistical error follows the idea of bootstrapping developed in \cite{Ganesh2022a} identically without modification.

\subsection{Adaptive hierarchy selection procedure and \gls{cmlmc}-gradient descent algorithm}
We discuss in this section how to select the parameters of the \gls{mlmc} hierarchy; namely the number of interpolation points $n$, the level-wise sample sizes $N_l$ and the number of levels $L$. 
The aim is to select these parameters such that a prescribed tolerance can be obtained on the gradient estimate $\objestgrad$.
In what follows, we drop the dependence on $z$ for notational simplicity, with the dependence being implied.
We propose here a minor variation of the framework presented in \cite[Section~5]{Ganesh2022a}.
An adaptive strategy was proposed therein for the selection of the hierarchy parameters $n$, $L$ and $N_l$ for any statistic $s_{\tau}$, the \gls{mse} of whose estimator $\hat{s}_{\tau}$ could be bounded by a linear combination of \glsplural{mse} on $\Phie$ and its derivatives:
\begin{align}
\mse{\hat{s}_{\tau}} \leq c_0 \mse{\Phie} + c_1 \mse{\Phieone} + c_2 \mse{\Phietwo} , \quad c_0, c_1, c_2 > 0.
\end{align}
We first note that the same hierarchy adaptivity procedure extends trivially to any linear combination of \glsplural{mse} of $\Phie$, $\Psiek$, and their derivatives. 
Specifically, this includes the case of the \gls{mse} on the gradient $\objestgrad$ in Eq.~\eqref{eq:grad_error_bound}.
In addition, each of the \glsplural{mse} on the parametric expectations in Eq.~\eqref{eq:grad_error_bound} can be split into its three error contributions, similar to Eq.~\eqref{eq:err_three_split}, leading to the following error estimator for $\mse{\objestgrad(w)}$:
\begin{align}
\mse{\objestgrad} &= \mse{\Phieone} + \sum_{k=1}^d \mse{\Psiekone} \nonumber\\
&\leq \underbrace{\left((\errest^{\Phi}_i)^2 + \sum_{k=1}^d(\errest^{\Psi_k}_i)^2 \right)}_{\text{Squared interpolation error}} + \underbrace{\left((\errest^{\Phi}_b)^2 + \sum_{k=1}^d(\errest^{\Psi_k}_b)^2 \right)}_{\text{Squared bias error}}+ \underbrace{\left((\errest^{\Phi}_s)^2 + \sum_{k=1}^d(\errest^{\Psi_k}_s)^2 \right)}_{\text{Squared statistical error}}.\label{eq:err_est_all_par_exp}
\end{align}
Here, $\errest^{\Phi}_{i}$, $\errest^{\Phi}_{b}$ and $\errest^{\Phi}_{s}$ denote the interpolation, bias and statistical error estimators corresponding to $\mse{\Phieone}$.
Once in the above form, the procedure described in \cite{Ganesh2022a} for adapting the hierarchy parameters $n$, $L$ and $N_l$ for linear combinations of \glsplural{mse} can be extended trivially to the current case when combined with the modifications proposed in Section~\ref{sec:error_mod}.

Lastly, we comment that the above adaptive procedure is carried out within the framework of the \gls{cmlmc} algorithm presented in \cite{Ganesh2022a}.
The \gls{cmlmc} algorithm works by first simulating a small ``screening'' hierarchy with relatively few samples and levels.
The algorithm then adapts the hierarchy parameters with respect to a decreasing set of tolerances, of which the target tolerance is the final one.
The optimal parameters for a given tolerance in the sequence are computed based on estimates obtained from the optimal hierarchy for the previous tolerance, or the initial ``screening'' hierarchy.
In this way, the \gls{mlmc} estimator becomes robust to large variations in the estimates produced by an initial screening hierarchy.

We now possess all the ingredients required to tailor Algorithm~\ref{alg:opti_mlmc_demo} to the specific case in which an \gls{mlmc} procedure is combined with a \gls{cmlmc} algorithm to estimate the gradient up to a prsecribed tolerance. 
The algorithm is detailed below, and differs from Algorithm~\ref{alg:opti_mlmc_demo} in that the first estimate of the gradient is computed based on a screening hierarchy, and that successive gradients are computed such that the \gls{mse} on the gradient satisfies a tolerance equal to a fraction of the gradient magnitude from the previous iteration; namely, the right-hand side of Eq.~\eqref{eq:gradient_error_bound_assumption} is estimated using $\hat{\obj}^{j-1}_w(w_{j-1})$.
Another key difference to note is that in contrast to \gls{cmlmc} algorithm described in \cite{Ganesh2022a}, the screening hierarchy used to compute first estimates for the design $z_j$ is the optimal hierarchy used to accurately estimate the gradient for the design $z_{j-1}$.
In addition, the gradient at the first design point $z_0$ is estimated using an initial small fixed hierarchy.\\

\begin{algorithm}[H]
\begin{algorithmic}
  \STATE Input: Initial design $z_0$, iterate $j=0$, tolerance $0<\epsilon <1$, step size $\alpha > 0$ and $\eta > 0$.
  \STATE Set residual $r = \epsilon+1$
  \WHILE{$r > \epsilon$}
    \STATE \textbf{if} $j = 0$ $\{$ Simulate screening hierarchy $\}$ 
    \STATE \textbf{else} $\bigg\{$ Start \gls{cmlmc} from the optimal hierarchy for $z_{j-1}$; Simulate \gls{cmlmc} adapting hierarchy such that $\mse{\objestgrad(\cdot,z_j)} \leq \eta^2 \norm{\hat{\obj}^{j-1}_{w}(w_{j-1})}_{\sltwo}^2$ $\bigg\}$     
    \STATE Compute minimiser $\theta_j \in \argmin_{\theta \in \Theta} \objest(\theta,z_j) = \Phie(\theta,z_j)$
    \STATE Compute gradient $\objt_{z}(\theta_j, z_j) = \Psie'(\theta_j;z_j) + 2 \kappa (z_j-z_{ref})$
    \STATE Compute gradient step $z_{j+1} =  z_j - \alpha \objt_{z}(\theta_j, z_j)$ and $\objestgrad(w_j) = \left(\objest_\theta(\theta_j,z_j)=0, \objt_z(\theta_j,z_j)\right)$
    \STATE Set residual $r = \norm{\objestgrad(w_{j})}_{\sltwo}^2 / \norm{\tilde{\nabla} \hat{\obj}^{0}(w_{0})}_{\sltwo}^2$
    \STATE Update $j \leftarrow j + 1$
  \ENDWHILE
\end{algorithmic}
\caption{\gls{cmlmc}-gradient descent \gls{ouu} algorithm}
\label{alg:opti_cmlmc}
\end{algorithm}

\section{Numerical results}\label{sec:results}
\subsection{FitzHugh Nagumo oscillator}\label{sec:fhn}
To demonstrate the optimisation framework, we use the FitzHugh--Nagumo system described in \cite{FitzHugh1961a} and \cite{nagumo1962active}.
The FitzHugh--Nagumo model is a two dimensional simplification of the Hodgkin-Huxley model introduced by \cite{Hodgkin1952a}, which was originally proposed in the field of neuroscience to model the phenomenon of spiking neurons.
The dynamical equations read as follows:
\begin{align}
\left[ \begin{matrix} \dot{v} \\ \dot{w} \end{matrix} \right] = \left[ \begin{matrix} v - \frac{v^3}{3} - w +I \\ \zeta \left(v + a - b w\right) \end{matrix} \right], \quad \left[ \begin{matrix} v(t=0)\\w(t=0) \end{matrix} \right] = \left[ \begin{matrix} v^0\\w^0 \end{matrix} \right], \quad t \in [0,T],
\end{align}
where $[v(t), w(t)]^T \in \setR^2$ denotes the state variables and $a$, $b$, $\zeta$ and $I$ denote system parameters. 
Fig.~\ref{fig:osc_nullcline} shows a phase-space plot containing the $v$ and $w$-nullclines for a nominal value of the system parameters. 
The oscillator enters a limit cycle for parameter values such that the intersection of the two nullclines lies in the interval $v \in [-1,1]$, indicated by the black lines.
If the intersection lies exterior to this interval, then the oscillator eventually reaches the intersection and remains at a constant value of $v$ and $w$. 
Although initially proposed to model neuron behaviour, the FitzHugh--Nagumo model has seen widespread use in modelling wave phenomena in excitable media.
Examples include blood coagulation \cite{ermakova2005blood, lobanov2005effect} and cardio-electrophysiological phenomena \cite{breiten2014riccati}, wherein the optimal control of the model plays an important role in the application.
The reader is referred to \cite{uzunca2017optimal} for an overview of existing work on the modelling applications and optimal control of the FitzHugh--Nagumo system.

\begin{figure}[h]
\centering
\includegraphics[scale=0.6]{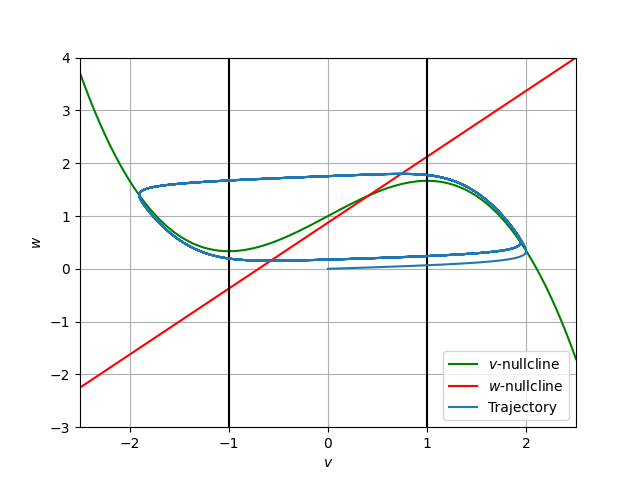}
\caption{FitzHugh--Nagumo oscillator dynamics}
\label{fig:osc_nullcline}
\end{figure}

In this work, we study the forced FitzHugh--Nagumo system:
\begin{align}
\left[ \begin{matrix} \dot{v} \\ \dot{w} \end{matrix} \right] = \left[ \begin{matrix} v - \frac{v^3}{3} - w +I +\sigma\dot{W}_1 \\ \zeta \left(v + a - b w\right) +\sigma\dot{W}_2 \end{matrix} \right], \quad \left[ \begin{matrix} v(t=0)\\w(t=0) \end{matrix} \right] = \left[ \begin{matrix} v^0\\w^0 \end{matrix} \right], \quad t \in [0,T], \label{eq:fhn_continuous}
\end{align}
where $\dot{W}_1$ and $\dot{W}_2$ are ``formal'' derivatives of standard Brownian paths and $\sigma = 0.01$ controls the noise strength.
To study the behaviour of the system, we propose the following \gls{qoi}:
\begin{align}
\qoi = \frac{1}{T} \int_0^T v^2(t) dt.
\end{align}
We are interested in minimising an objective function of the form in Eq.~\eqref{eq:opt_form_combined}, where we seek to minimise the \gls{cvar} with significance $\tau = 0.7$.
We denote by $z = [a, b, \zeta, I]^T$ the vector of design parameters with respect to which we want to carry out the optimisation, and seek to penalise deviations from the design $z_{ref}= [0.8,0.7,0.08,1.0]$.

We discretise the interval $[0,T]$ using a hierarchy of uniform grids $t_j = j \Delta t_l, j \in \{0,1,...,N_{T,l}\}$, with $\Delta t_l = T/N_{T,l}$ and $N_{T,l} = N_{T,0} 2^l$. 
We set $T=10$ and $N_{T,0} = 20$, and consider an Euler-Maruyama discretisation of Eq.~\eqref{eq:fhn_continuous}.
Using the notation $v^l_n$ to denote the approximation of $v(t_{n})$ at level $l$, the discretised system then reads:
\begin{align}
\left[ \begin{matrix} v^l_{n+1}\\ w^l_{n+1} \end{matrix} \right] &= \left[ \begin{matrix} v^l_{n}\\ w^l_{n} \end{matrix} \right] + \Delta t_l \left[ \begin{matrix} v^l_n - \frac{(v^l_n)^3}{3} - w^l_n +I \\ \zeta \left(v^l_n + a - b w^l_n \right) \end{matrix} \right] + \sigma \sqrt{\Delta t_l}\left[ \begin{matrix} \xi^l_{1,n} \\ \xi^l_{2,n} \end{matrix} \right] , \\
\left[ \begin{matrix} v^l_0\\w^l_0 \end{matrix} \right] &= \left[ \begin{matrix} v^0\\w^0 \end{matrix} \right], \quad n \in \{0,...,N_{T,l}-1\},
\end{align}
where $\xi^l_{1,n}$ and $\xi^l_{2,n}$ are independently drawn realisations of standard normal random variables.
The quantity of interest that we study is the following time average:
\begin{align}
\qoi = \frac{1}{T} \int_0^T v^2(t) dt \approx \sum_{n=0}^{N_{T,l}-1}\left( \frac{(v^l_n)^2 + (v^l_{n+1})^2}{2}\right) \frac{\Delta t_l}{T} \eqqcolon \qoi_l.
\end{align}
To compute the sensitivities $\qoi_{z,l}$, we utilize the method of adjoints.
We consider the corresponding adjoint variables $\lambda^l_n$ and $\nu^l_n$ corresponding to $v^l_n$ and $w^l_n$, $n \in \{1,...,N_{T,l}\}$ respectively.
The adjoint equation reads as follows:
\begin{align}
\left[ \begin{matrix} \lambda^l_n \\ \nu^l_n \end{matrix} \right] &= \left[ \begin{matrix} \lambda^l_{n+1} \\ \nu^l_{n+1} \end{matrix} \right] + \Delta t_l \left( \left[ \begin{matrix} (1-(v^l_n)^2) & \zeta \\ -1 & -\zeta b \end{matrix} \right] \left[ \begin{matrix} \lambda^l_{n+1} \\ \nu^l_{n+1} \end{matrix} \right] + \left[ \begin{matrix} \frac{2 v^l_n}{T} \\ 0 \end{matrix} \right] \right), \\
\left[ \begin{matrix} \lambda^l_{N_{T,l}} \\ \nu^l_{N_{T,l}} \end{matrix} \right] &= \Delta t_l \left[ \begin{matrix} \frac{v^l_n}{T} \\ 0 \end{matrix} \right], \quad n \in \{1,...,N_{T,l}-1\}.
\end{align}
The reader is referred to Appendix~\ref{sec:proof-oscillator} for the details of the derivation.

Once the adjoint equation is solved backwards in time, the approximation $\qoi_{z,l}$ of the sensitivities $\qoi_{z}$ at level $l$ can then be obtained as follows:
\begin{equation}
\begin{aligned}
\qoi_{a,l} = \sum_{n=0}^{N_{T,l}-1} \Delta t_l \zeta \nu^l_{n+1}, & \qquad \qoi_{b,l} = -\sum_{n=0}^{N_{T,l}-1} \Delta t_l \zeta w^l_n \nu^l_{n+1}, \\
\qoi_{I,l} =\sum_{n=0}^{N_{T,l}-1} \Delta t_l \lambda^l_{n+1} , & \qquad  \qoi_{\zeta,l} = \sum_{n=0}^{N_{T,l}-1} \Delta t_l (v^l_n + a - bw^l_n) \nu^l_{n+1}.
\end{aligned}\label{eq:adjoint_osc}
\end{equation}

To demonstrate the performance of Algorithm~\ref{alg:opti_cmlmc}, we assess the performance individually of its two components; firstly, the performance of the \gls{cmlmc} algorithm, the error estimation procedure and the adaptive strategy described in Section~\ref{sec:grad_est_mlmc} for accurately estimating the gradient for a given design, and secondly, the gradient based optimisation procedure described in Algorithm~\ref{alg:opti_cmlmc}.
We first assess the performance of the \gls{cmlmc} algorithm and adaptive strategy.
We remark that the solution of the forward and adjoint problems, as well as the \gls{cmlmc} procedure, are implemented within the XMC software library \cite{ExaQUte_XMC}, which we use for the simulations presented herein.

We seek to accurately estimate the gradient $\objgrad(\cdot,z_0)$ using the estimator $\hat{\obj}_{w}(\cdot,z_0)$, where $z_0 = [0.7, 0.8, 0.08, 1.0]$ and we set $\tau = 0.70$ for the significance of the \gls{cvar}.
The gradient and gradient error are estimated using the \gls{mlmc} procedure described in Sections~\ref{sec:grad_est_mlmc}.
To assess the reliability of the error bound derived in Proposition~\ref{prop:gradient_error_bound}, we run a reliability study wherein we adapt the parameters of the \gls{mlmc} hierarchy to attain a prescribed tolerance on $\mse{\hat{\obj}_{w}(\cdot,z_0)}$. 
We run the \gls{mlmc} algorithm 20 times for each tolerance tested and compare the estimated error to the true error obtained using a reference gradient computed using a Monte Carlo estimator with $2 \times 10^5$ samples and $2 \times 10^4$ time steps.
Specifically, we are interested in assessing the tightness of the inequality in Eq.~\eqref{eq:err_est_all_par_exp}.

The resultant plot is shown in Fig.~\ref{fig:fhn_reliability}.
Three errors are plotted in Fig.~\ref{fig:fhn_reliability}; namely, the true error on the gradient, defined in the $\linf$ sense, corresponding to the term on the leftmost side of Eq.~\eqref{eq:err_est_all_par_exp}, the square root of the \gls{mse} estimate on the gradient, produced by the optimally calibrated \gls{mlmc} hierarchy, corresponding to the term on the rightmost side of Eq.~\eqref{eq:err_est_all_par_exp}, and the true error on the gradient evaluated at the point $(\theta_0,z_0)$, where $\theta_0$ corresponds to the $70\%$-\gls{var} for the design $z_0$.
The true errors are computed with respect to a reference solution computed using $2 \times 10^5$ samples and $2 \times 10^4$ time steps.
As can be seen from the figure, the \gls{mse} estimator provides a tight bound on the true error on the parametric expectations.
However, the true error on the gradient in the $\linf$ sense is much larger than the true pointwise error.
This is a natural consequence of using the $\linf$-norm over the entire interval $\Theta$ to define the \gls{mse}, as compared to using the pointwise error. 
Controlling the \gls{mse} error in an $\linf$ sense, as defined in Eq.~\eqref{eq:grad_error_def}, is necessitated by the error accuracy condition in Eq.~\eqref{eq:gradient_error_bound_assumption}, in order to ensure Q-linear convergence of Algorithm~\ref{alg:opti_cmlmc}.

Fig.~\ref{fig:fhn_complexity} shows the complexity behaviour of the \gls{mlmc} estimator calibrated using the \gls{cmlmc} algorithm.
We compute the cost required to obtain the final optimal hierarchy for a given tolerance $\epsilon^2$ on $\mse{\hat{\obj}_{w}(\cdot,z_0)}$.
As can be seen from the figure, the cost grows as $\epsilon^{-2}$, which is the theoretically predicated best case performance for the \gls{mlmc} estimator.
For comparison, we also plot the estimated cost of a comparable Monte Carlo estimator, as well as the expected cost growth rate for the case of the first order time discretisation used here.
The Monte Carlo reference cost is computed as described in \cite{Ganesh2022a}.

\begin{figure}[h]
  \centering
  \begin{subfigure}{0.48\textwidth}
    \centering
    \includegraphics[width=\textwidth]{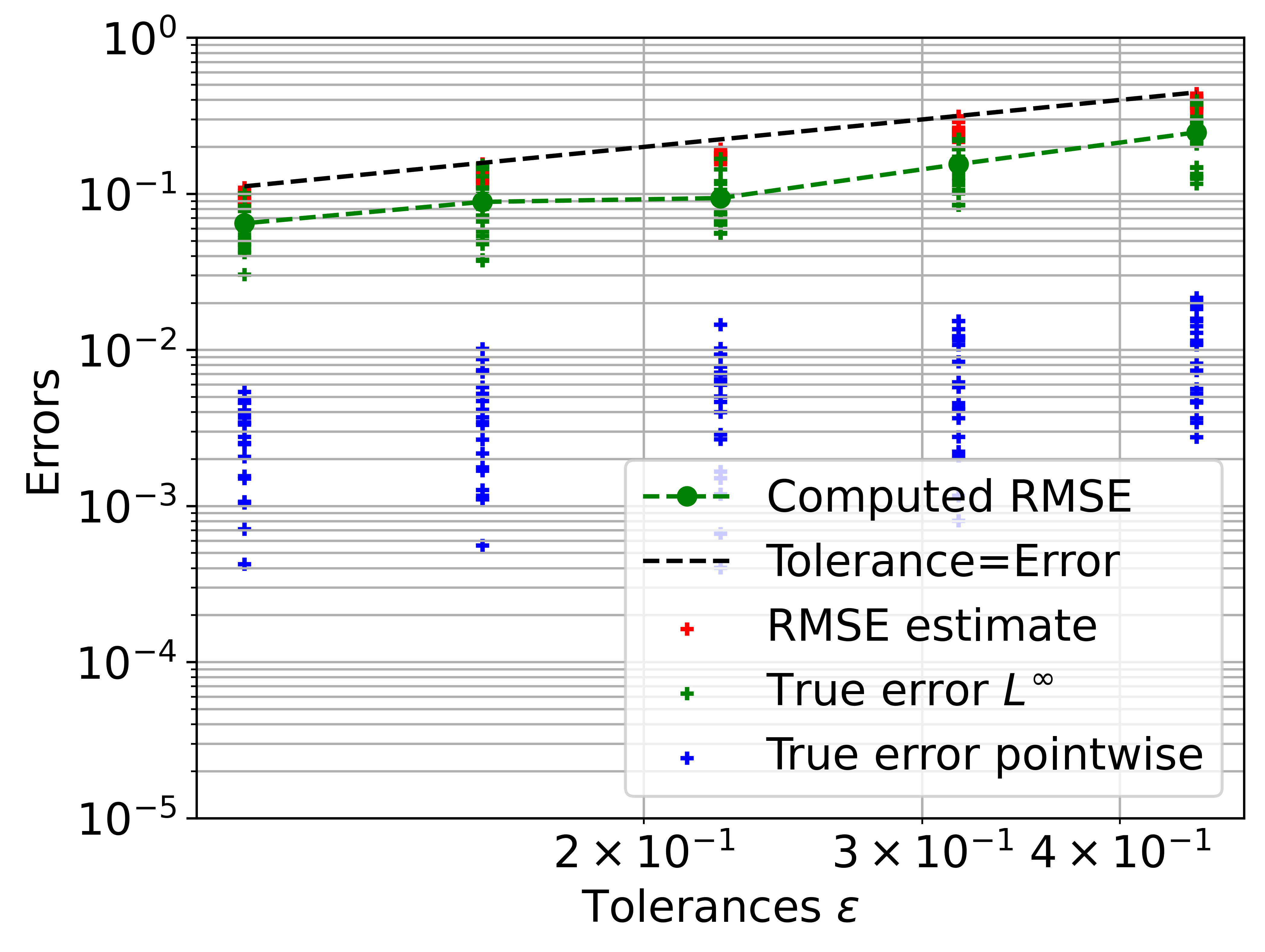}
    \caption{Reliability of error estimator}
    \label{fig:fhn_reliability}
  \end{subfigure}
  \begin{subfigure}{0.48\textwidth}
    \centering
    \includegraphics[width=\textwidth]{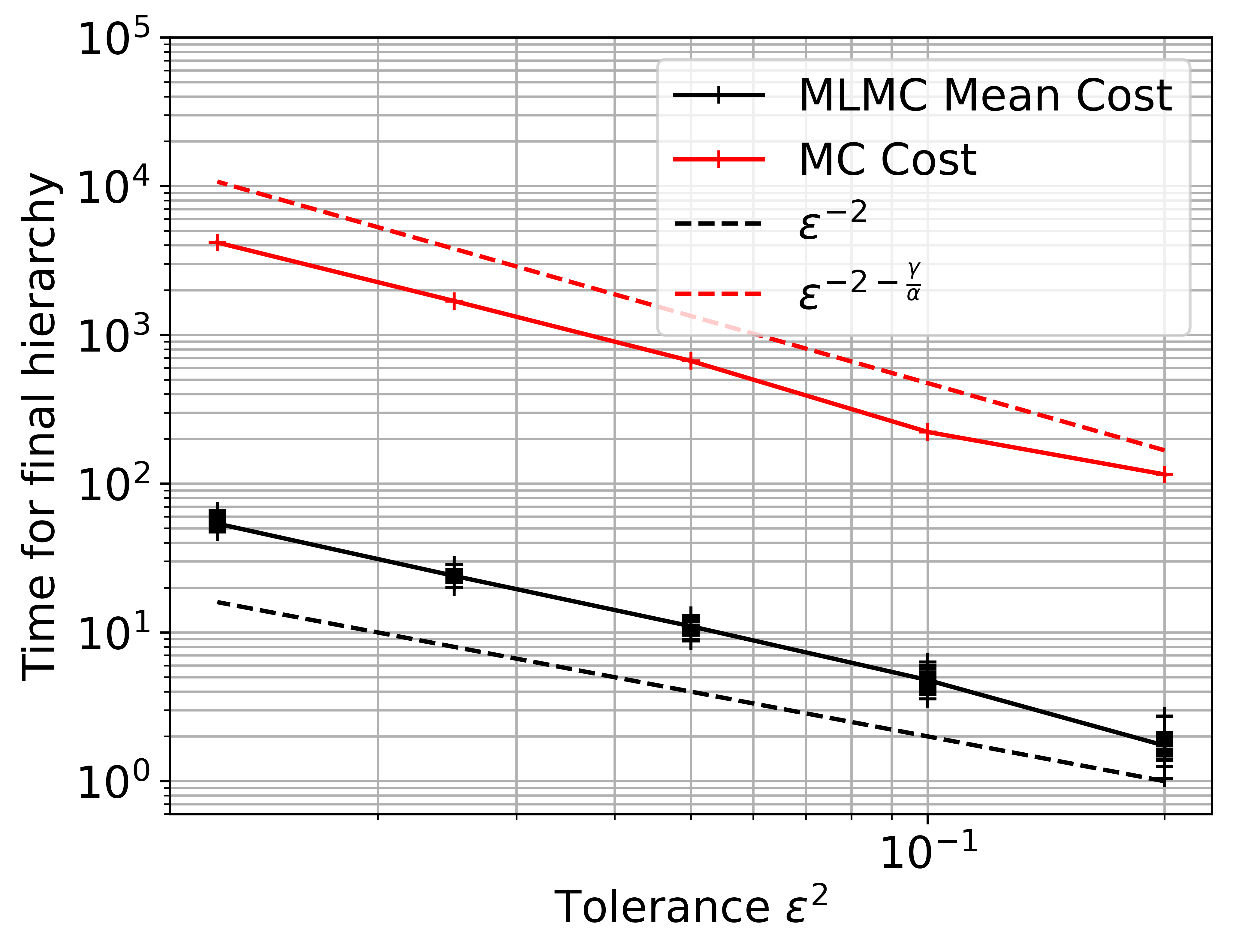}
    \caption{Complexity behaviour}
    \label{fig:fhn_complexity}
  \end{subfigure}
  \caption{Error estimator performance for the \gls{cmlmc} estimator of the gradient for the FitzHugh--Nagumo system}
  \label{fig:fhn_performance}
\end{figure}

We now examine the performance of the gradient descent algorithm proposed in Section~\ref{sec:conv_novel_alg}.
We are interested in solving the minimisation problem given in Eq.~\eqref{eq:opt_form_combined}, with $\tau = 0.7$.
We utilize the framework of Algorithm~\ref{alg:opti_cmlmc}, with a tolerance $\epsilon = 0.01$ on the gradient ratio.
This implies that we stop the algorithm once the gradient magnitude has dropped to $1/100^{\text{th}}$ of its initial magnitude.
As an initial guess, we begin with the design $z_0 = [0.7, 0.8, 0.08, 1.0]$.
We also set $z_{ref} = [0.7, 0.8, 0.08, 1.0]$.
We combine the above with the \gls{cmlmc} algorithm detailed in \cite{Ganesh2022a} and detailed further in Section~\ref{sec:grad_est_mlmc}, with $\eta = 0.2$ on the relative error on the gradient.

We plot in Fig.~\ref{fig:fhn_obj_func} the value of the objective function for different iterations of the objective function. 
We observe Q-linear convergence in the number of iterations towards the final value, as predicted by Theorem~\ref{thm:approx_convergence}, although we cannot guarantee that the hypotheses of Theorem~\ref{thm:approx_convergence} are satisfied for this problem.
Fig.~\ref{fig:fhn_grad_ratio} shows the value of the gradient ratio $r$ for different iterations of the optimisation algorithm. 
We also observe that the gradient decreases Q-linearly.
Lastly, we plot in Fig.~\ref{fig:fhn_cdf} the \gls{cdf} of the output \gls{qoi} $\qoi(z_j,\cdot)$ computed at different iterations of the optimisation algorithm, as well as the predicted \gls{var} and \gls{cvar} values.
We observe that the \gls{cdf}, the \gls{var} and the \gls{cvar} all move left, reducing the mass in the right tail of the distribution.
Since we are minimising the \gls{cvar}, defined as the expectation of the random variable above the \gls{var}, this translates to moving the right tail of the distribution as much as possible to the left. 

\begin{figure}[h]
  \centering
    \begin{subfigure}{0.32\textwidth}
    \centering
    \includegraphics[width=\textwidth]{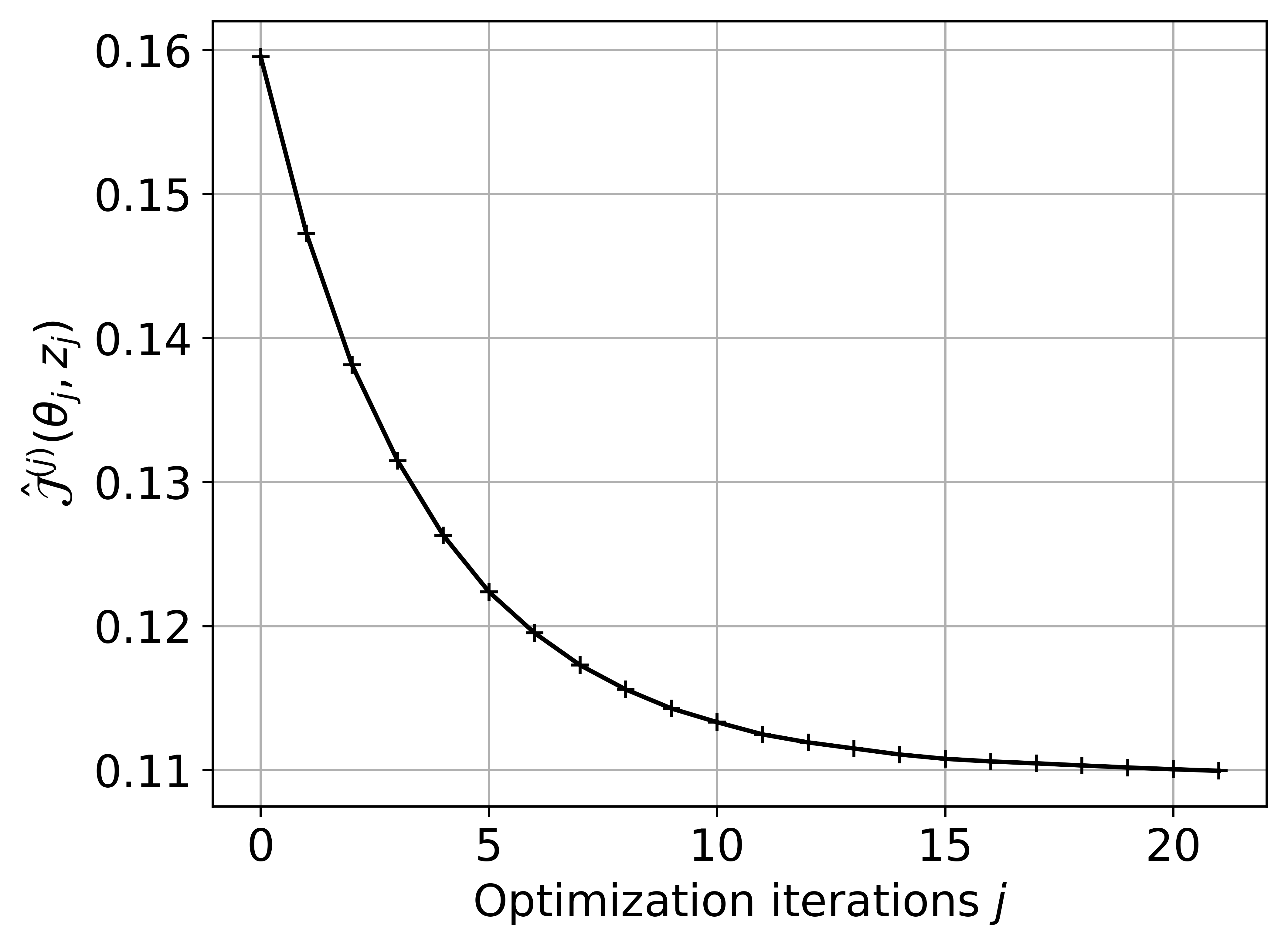}
    \caption{Objective function decay.}
    \label{fig:fhn_obj_func}
  \end{subfigure}
  \begin{subfigure}{0.32\textwidth}
    \centering
    \includegraphics[width=\textwidth]{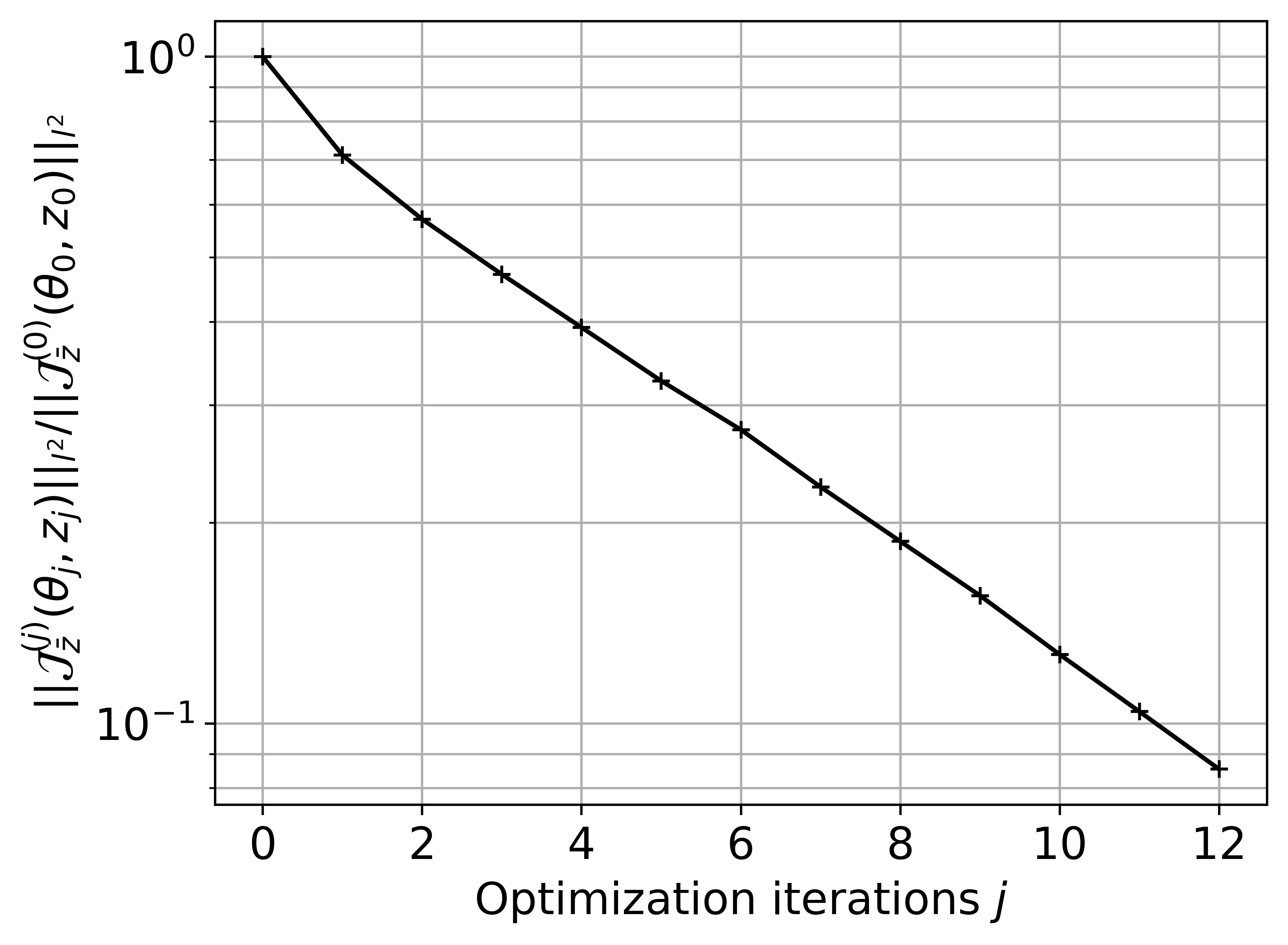}
    \caption{Gradient ratio decay.}
    \label{fig:fhn_grad_ratio}
  \end{subfigure}
  \begin{subfigure}{0.32\textwidth}
    \centering
    \includegraphics[width=\textwidth]{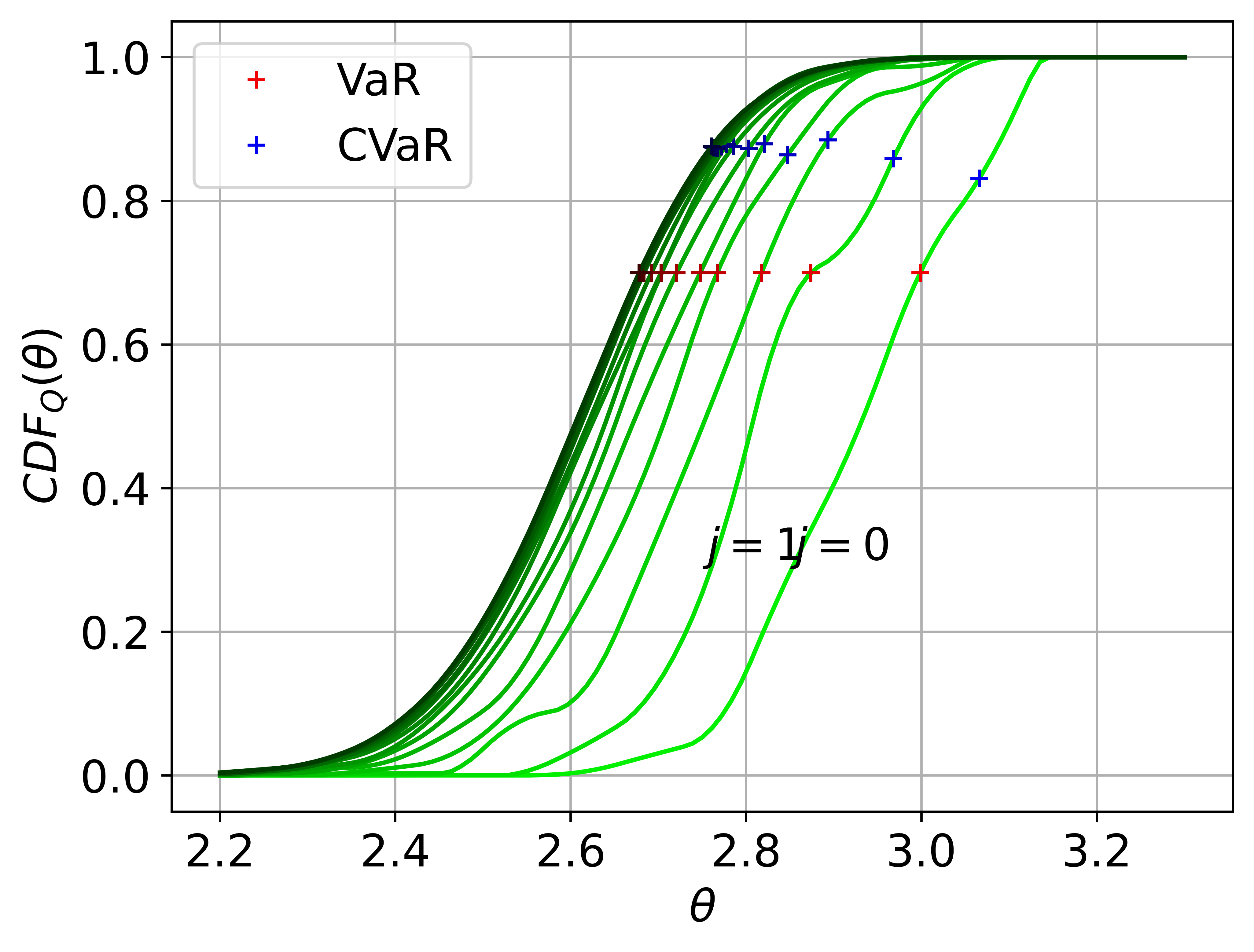}
    \caption{Change in \gls{cdf}.}
    \label{fig:fhn_cdf}
  \end{subfigure}
  \caption{Performance of Algorithm~\ref{alg:opti_cmlmc} over different iterations for the FitzHugh--Nagumo system}
  \label{fig:fhn_opti_performance}
\end{figure}

Fig.~\ref{fig:fhn_hierarchy} shows the optimal hierarchy produced by the \gls{cmlmc} algorithm at each iteration of the optimisation.
We observe that since the tolerance supplied to the \gls{cmlmc} algorithm is a fraction of the gradient magnitude, the optimally tuned hierarchy becomes larger for later iterations of the optimisation. 
In addition, Fig.~\ref{fig:fhn_opti_complexity} shows the cumulative cost required for the optimisation algorithm to reach a given gradient magnitude.
The cumulative cost at a given optimisation iteration is defined as the sum of costs of all optimal hierarchies until the current optimisation iteration.
Specifically, the cumulative cost is computed as $\sum_{i=0}^j \sum_{l=0}^L N_l^{(i)} ( \cost{\qoi_l} +  \cost{\qoi_{l-1}} )$, where $\{N_l^{(i)}\}_{l=0}^L$ denote the optimal level-wise sample sizes for the $i$\textsuperscript{th} optimisation iteration and $\cost{\qoi_l}$ denotes the average cost of simulating one sample of $\qoi_l$.
This cost is plotted versus the gradient magnitude.
We observe that after an initial pre-asymptotic regime, the cumulative cost grows as $\norm{\objestgrad(w_j)}_{\sltwo}^{-2}$, a rate comensurate with the use of an optimally tuned \gls{mlmc} hierarchy at each iteration tuned to obtain a tolerance proportional to $\norm{\objestgrad(w_j)}_{\sltwo}$.

\begin{figure}[h]
  \centering
  \begin{subfigure}{0.48\textwidth}
    \centering
    \includegraphics[width=\textwidth]{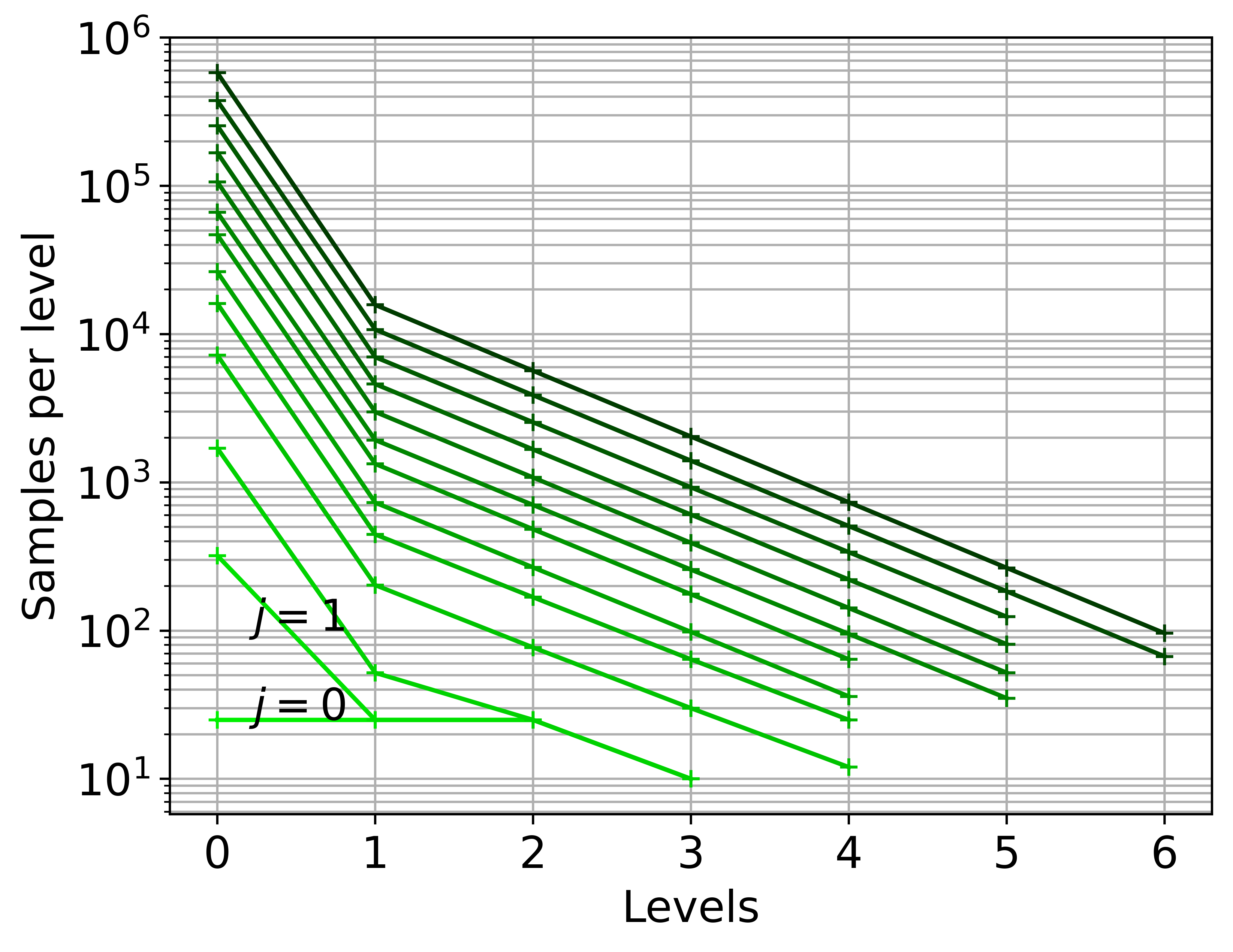}
	\caption{Level-wise sample sizes}
    \label{fig:fhn_hierarchy}
  \end{subfigure}
  \begin{subfigure}{0.48\textwidth}
    \centering
    \includegraphics[width=\textwidth]{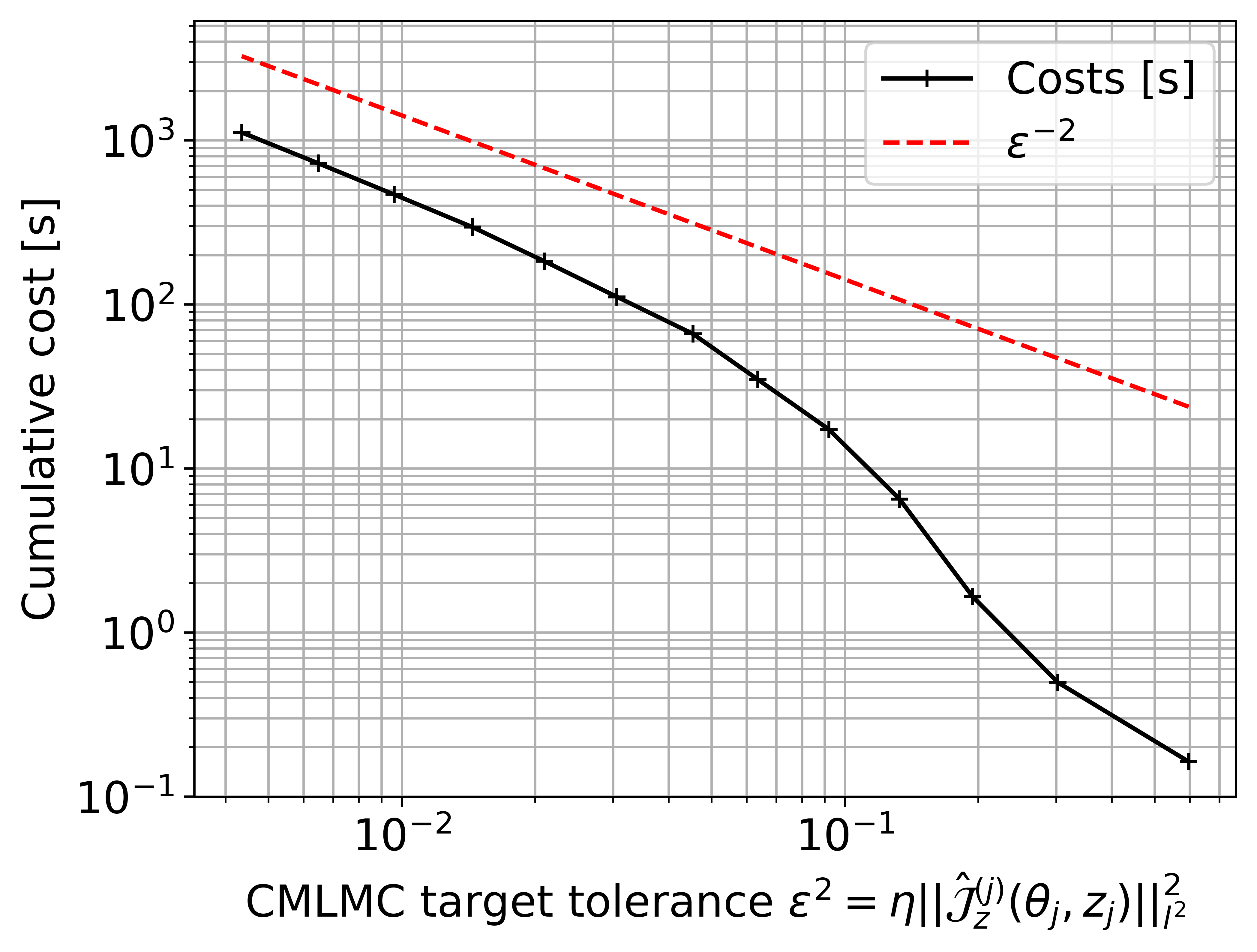}
	\caption{Cumulative cost}
    \label{fig:fhn_opti_complexity}
  \end{subfigure}
  \caption{Hierarchy of \gls{cmlmc} estimators and complexity behaviour of Algorithm~\ref{alg:opti_cmlmc} for different iterations for the FitzHugh--Nagumo system}
  \label{fig:fhn_comp_performance}
\end{figure}

\subsection{Pollutant transport problem}
We now apply the methodology to a more applied problem of practical relevance.
A problem of pollutant transport is studied, where the concentration of pollutant in a domain is modelled using a steady reaction-diffusion-advection equation.
We consider a square domain $D = (0,1) \times (0,1)$, with boundary $\partial D \coloneqq \Gamma_d \cup \Gamma_n$, where $\Gamma_d \coloneqq \{0\} \times (0,1)$ and $\Gamma_n \coloneqq \partial D \setminus \Gamma_d$.
We denote by $u : D \times \setR^9 \times \Omega\to \setR$ the concentration of the pollutant.
The concentration satisfies the following equation:
\begin{align}
-\nabla \cdot (\epsilon \nabla u(x,z,\omega)) + \mathbb{V}(x,\omega) \cdot \nabla u(x,z,\omega) = f(x) - B(x,z), \quad x \in D,
\end{align}
subject to the following boundary conditions:
\begin{align}
\epsilon \frac{\partial u}{\partial n}(x,z,\omega) &= 0, \quad x \in \Gamma_n, \quad \text{for }\measure-\text{a.e. } \omega \in \Omega\\
u(x,z,\omega) &= 0, \quad x \in \Gamma_d, \quad \text{for }\measure-\text{a.e. } \omega \in \Omega,
\end{align}
where $\epsilon > 0$ denotes a viscosity parameter.
$\mathbb{V}(x,\omega)$ is a random divergence-free velocity field defined as follows:
\begin{align}
\mathbb{V}(x,\omega) \coloneqq \left[ \begin{matrix}
b(\omega) - a(\omega) x_1 \\ a(\omega) x_2
\end{matrix}\right],
\end{align}
where $a \sim \mathcal{U}[4.95,5.05]$ and $b \sim \mathcal{U}[3.95,3.05]$ are uniformly distributed random variables, and $x_1$ and $x_2$ denote the components of $x$.
The source $f(x)$ is the sum of five Gaussian source terms:
\begin{align}
f(x) = \sum_{i=1}^5 s_i \exp\left(- \frac{(x-
\mu_i)^T(x-\mu_i)}{2\sigma_i^2}\right),
\end{align}
where the values of $s_i$, $\mu_i$ and $\sigma_i$ are given in Table~\ref{tab:source_params}.
The sink term $B(x,z)$ is defined as follows:
\begin{align}
B(x,z) = \sum_{k=1}^9 z^k \exp\left(- \frac{(x-p_k)^T(x-p_k)}{2\sigma^2}\right),
\end{align}
where the locations $p_k$ are defined as $p_k = (0.25i,0.25j), i,j \in \{1,2,3\}, k = 3(i-1)+j$, $\sigma = 0.05$, and $z^k$ denotes the $k^{\text{th}}$ component of $z \in \setR^9$.
We are interested in studying the distribution of the random \gls{qoi} $\qoi$, defined as follows:
\begin{align}
\qoi(z,\omega) \coloneqq \frac{\kappa_s}{2} \int_D u^2(x,z,\omega) dx,
\end{align}
with $\kappa_s = 10^4$.

\begin{table}[ht]
\centering
\begin{tabular}{cccc}
	\toprule
	$i$ & $\mu_i $ & $\sigma_i$ & $s_i$\\
	\midrule
	$1$ & $ [0.55205319,0.65571641]^T $ & $0.0229487$ & $2.3220339$\\
	$2$ & $ [0.49379544,0.10950509]^T $ & $0.0205321$ & $1.7931427$\\
	$3$ & $ [0.13032797,0.57569277]^T $ & $0.0196891$ & $2.3522452$\\
	$4$ & $ [0.33868732,0.37971428]^T $ & $0.0212297$ & $2.2850373$\\
	$5$ & $ [0.27670822,0.15833522]^T $ & $0.0227373$ & $2.3194400$\\
	\bottomrule
\end{tabular}
\caption{Source term parameters for the pollutant transport problem}
\label{tab:source_params}
\end{table}

The problem is implemented using the FEniCS finite element software \cite{logg2012automated}.
The domain is discretised using a uniform triangular mesh with piecewise linear finite elements.
The resultant linear system is solved using a sparse direct solver \cite{MUMPS:1, MUMPS:2}.
The number of elements per side of the square domain varies as $32 \times 2^{l/2}, \; l \in \{0,1,...,L\}$, leading to a mesh size $h_l$ that varies as $h_l = h_0 \times 2 ^{-l}$.
An in-built automatic differentiation module within the FEniCS library is used to compute the sensitivities of the \gls{qoi} with respect to design parameters.
Once again, the XMC software library \cite{ExaQUte_XMC} is used to implement the \gls{cmlmc} procedure.

Similar to Section~\ref{sec:fhn}, we seek to examine both parts of the optimisation algorithm; namely the \gls{cmlmc} and the gradient based \gls{ouu} algorithm.
For the \gls{cmlmc}, we seek to accurately estimate $\hat{\obj}_w(\cdot,z_0)$, where $z_0 = [0.1]^9$, such that $\mse{\hat{\obj}_w(\cdot,z_0)}$ satisfies a prescribed tolerance.
Fig.~\ref{fig:pol_performance} shows the results of reliability and complexity studies conducted for the above parameters, similar to the one conducted for the FitzHugh--Nagumo system in Section~\ref{sec:fhn}.
For studying the reliability of the error estimators, we conduct 20 independent \gls{cmlmc} simulations for a given tolerance. 
For each simulation, we plot three errors; namely the true $\linf$ error on the gradient, the square root of the \gls{mse} estimate produced by our error estimation procedure described in Section~\ref{sec:grad_est_mlmc}, and the true pointwise error on the gradient, computed by evaluating the parametric expectations $\hat{\obj}_w(\cdot,z_0)$ for the gradient at $\theta_0$, the \gls{var} corresponding to the design $z_0$.
The reference value of the gradient is computed by first running 20 simulations for a tolerance that is half of the finest tested tolerance, and averaging over the gradient estimates produced by these simulations.
Similar to before, we find that although our novel error estimators provide a tight bound on the true $\linf$ error of the gradient, the $\linf$ error on the gradient is significantly larger than the error on the gradient evaluated at $\theta_0$.
Fig.~\ref{fig:pol_complexity} presents the complexity results of the \gls{cmlmc} algorithm.
The cost to compute the optimal hierarchy for a given tolerance $\epsilon^2$ on $\mse{\objestgrad(\cdot, z_0)}$ is plotted versus the tolerance, for each of the 20 \gls{cmlmc} simulations at a given tolerance, in addition to their sample average value. 
In addition, the theoretical cost growth rate of a comparable Monte Carlo estimator is shown, as well as the estimated cost of the estimator for reference and comparison. 
The Monte Carlo reference cost is computed as described in \cite{Ganesh2022a}.
As can be seen from the figure, the complexity follows the theoretically predicted complexity $\epsilon^{-2}$.

\begin{figure}[ht]
  \centering
  \begin{subfigure}{0.48\textwidth}
    \centering
    \includegraphics[width=\textwidth]{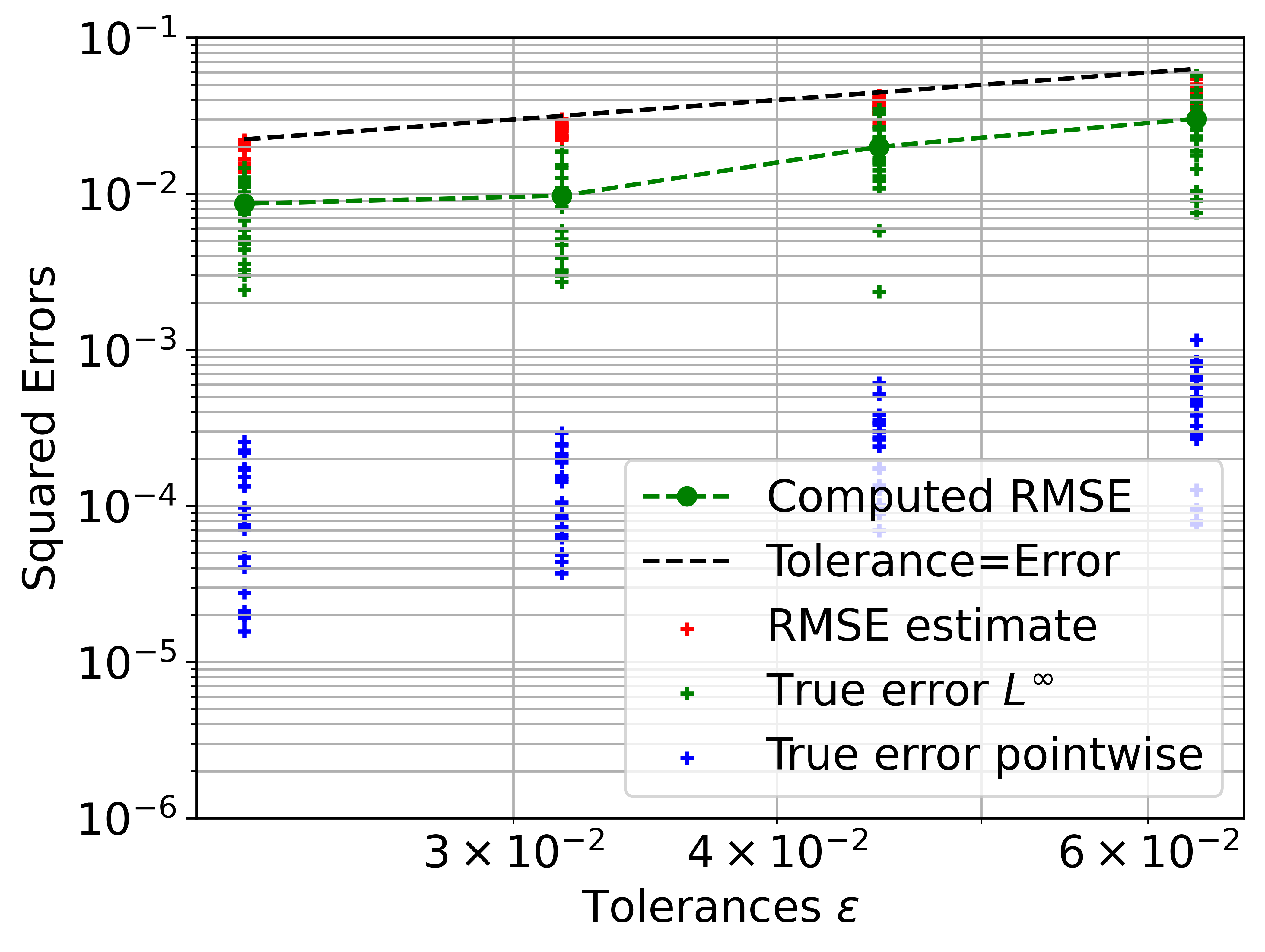}
    \caption{Reliability of error estimator}
    \label{fig:pol_reliability}
  \end{subfigure}
  \begin{subfigure}{0.48\textwidth}
    \centering
    \includegraphics[width=\textwidth]{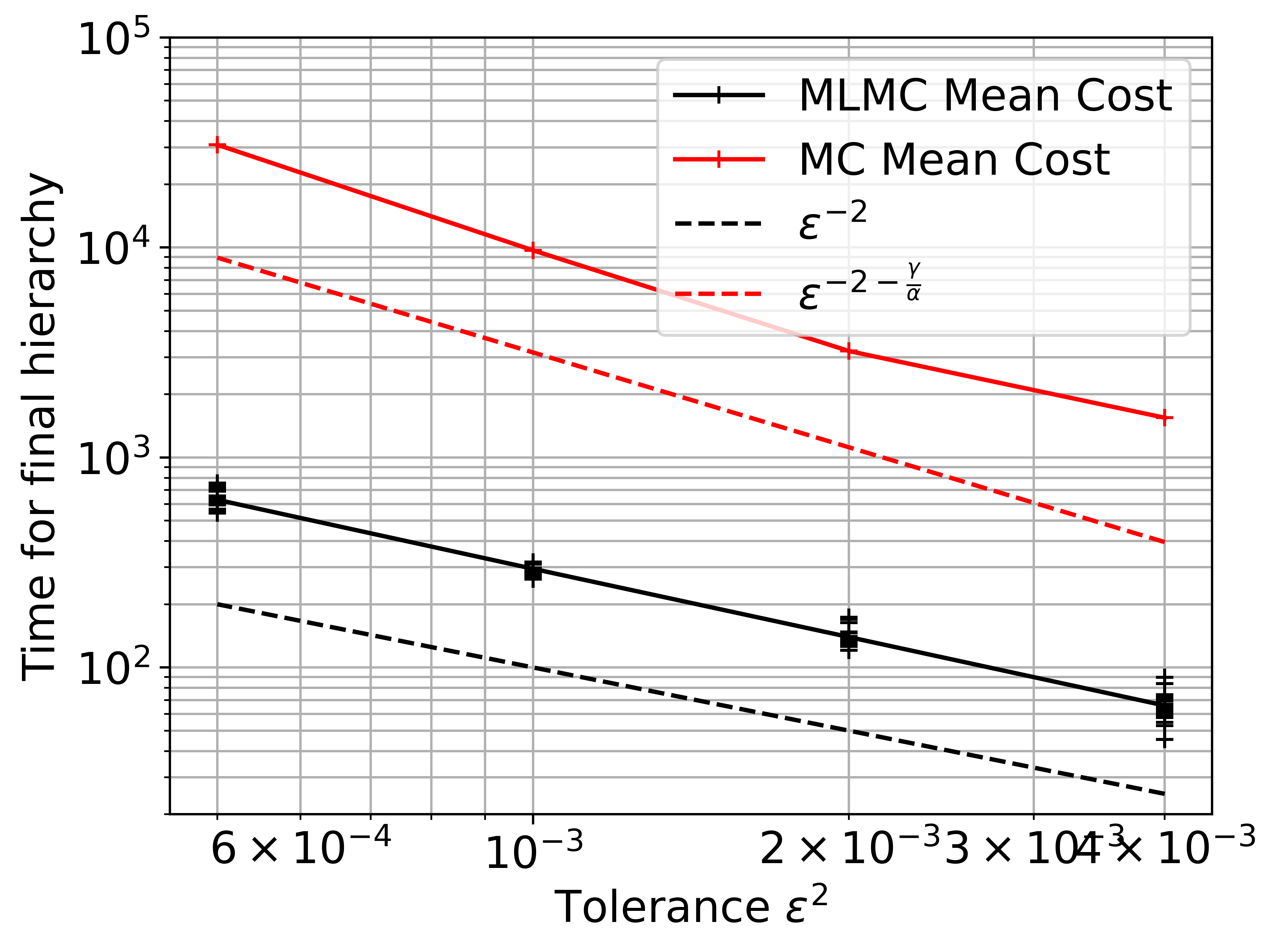}
    \caption{Complexity behaviour}
    \label{fig:pol_complexity}
  \end{subfigure}
  \caption{Error estimator performance for the pollutant transport problem}
  \label{fig:pol_performance}
\end{figure}

For the \gls{ouu}, we wish to minimise an objective function of the form in Eq.~\eqref{eq:opt_form_combined}, with $z_{ref} = 0$ and for significance of $\tau = 0.7$.
This implies that we seek to minimise the \gls{cvar} while also minimising the amplitude of the controlled sinks.
We utilise Algorithm~\ref{alg:opti_cmlmc}, starting from a design $z_0 = [0.1]^9$, and halt the optimisation once a gradient ratio of $r=0.08$ has been achieved.
In Fig.~\ref{fig:pol_simulation}, we show the source field $f(x)$, the control field $B(x,z^*)$ and the solution $u(x,z^*,\omega)$ for the mean conditions $a(\omega) = 4$ and $b(\omega) = 5$ at the optimal control $z^*$ obtained by solving problem~\eqref{eq:opt_form_combined}.

\begin{figure}[H]
\centering
\begin{subfigure}{.31\textwidth}
    \centering
    \includegraphics[width=\textwidth]{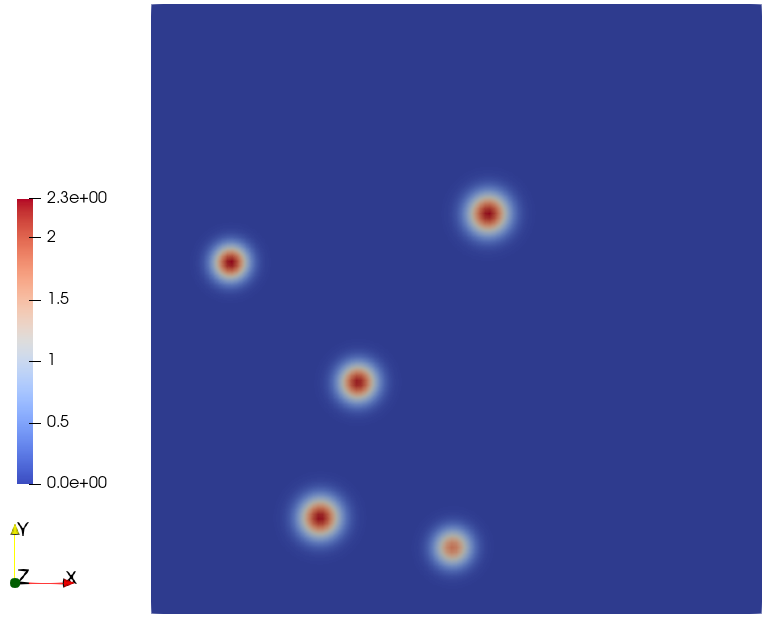}
    \caption{$f(x)$}
\end{subfigure}
\begin{subfigure}{.31\textwidth}
    \centering
    \includegraphics[width=\textwidth]{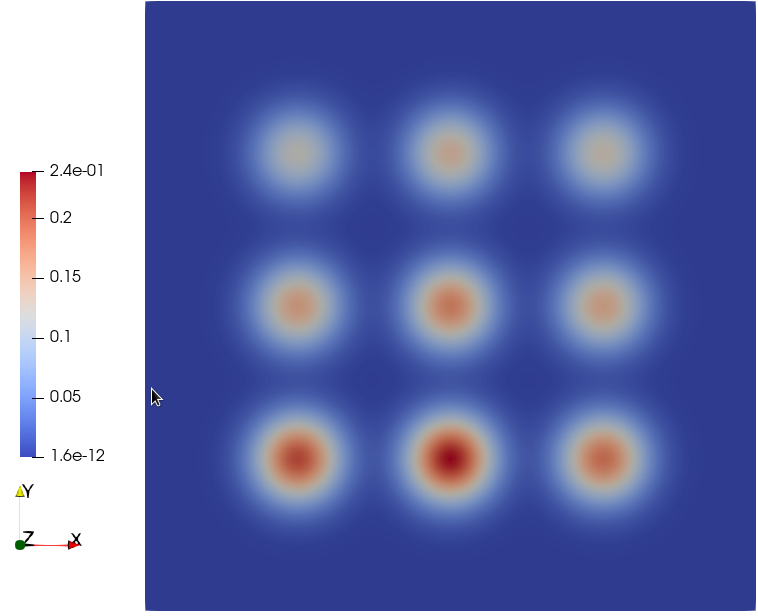}
    \caption{$B(x,z^*)$}
\end{subfigure}
\begin{subfigure}{.31\textwidth}
    \centering
    \includegraphics[width=\textwidth]{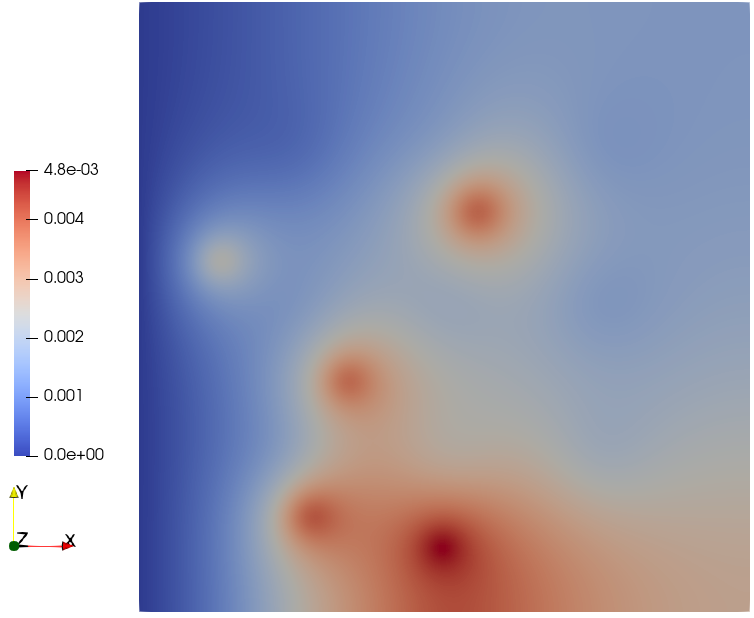}
    \caption{$u(x,z^*,\omega)$}
\end{subfigure}
\caption{Source, control and solution fields for the pollutant transport problem for $a(\omega) = 4$ and $b=5(\omega)$, and for $x \in D$}
\label{fig:pol_simulation}
\end{figure}

Fig.~\ref{fig:pol_obj_func} shows the decay of the objective function towards its final value. 
We once again observe Q-linear convergence in the optimisation counter $j$, as predicted by Theorem~\ref{thm:approx_convergence}.
In addition, we plot in~\ref{fig:fhn_grad_ratio} the gradient ratio for different iterations of the optimisation, which also decreases Q-linearly in the iteration counter $j$.
Fig.~\ref{fig:pol_cdf} shows the \gls{cdf} of the output \gls{qoi} $\qoi(z_j,\cdot)$ for different iterations $j$ of the optimisation algorithm, along with the estimated \gls{var} and \gls{cvar}.
The \gls{cdf}, the \gls{var} and the \gls{cvar} all move left as before in Section~\ref{sec:fhn}, which translates to moving the right tail of the distribution as much as possible to the left. 

\begin{figure}[h]
  \centering
    \begin{subfigure}{0.32\textwidth}
    \centering
    \includegraphics[width=\textwidth]{img/optimization_objfun_pol.png}
    \caption{Objective function decay.}
    \label{fig:pol_obj_func}
  \end{subfigure}
  \begin{subfigure}{0.32\textwidth}
    \centering
    \includegraphics[width=\textwidth]{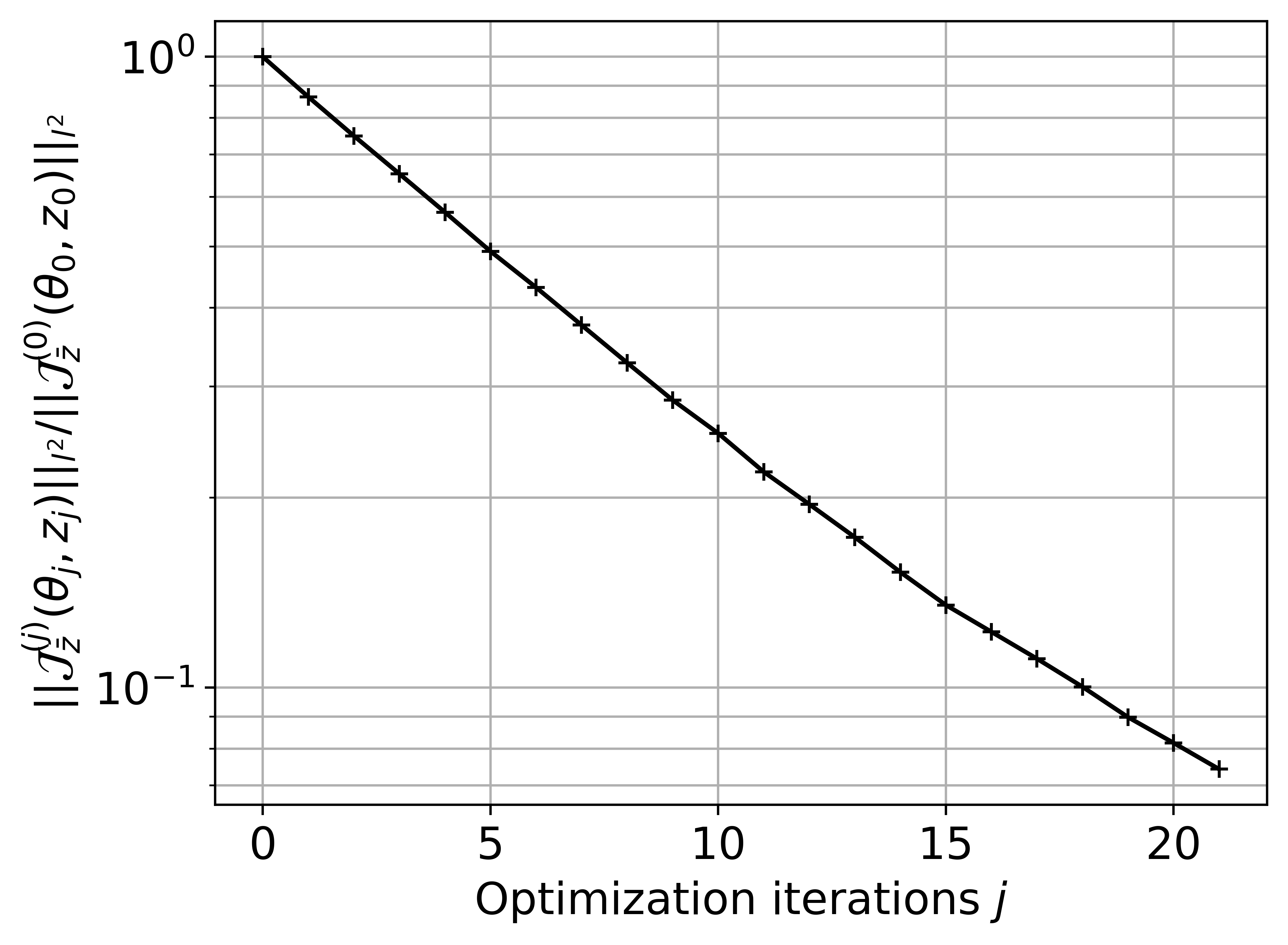}
    \caption{Gradient ratio decay.}
    \label{fig:pol_grad_ratio}
  \end{subfigure}
  \begin{subfigure}{0.32\textwidth}
    \centering
    \includegraphics[width=\textwidth]{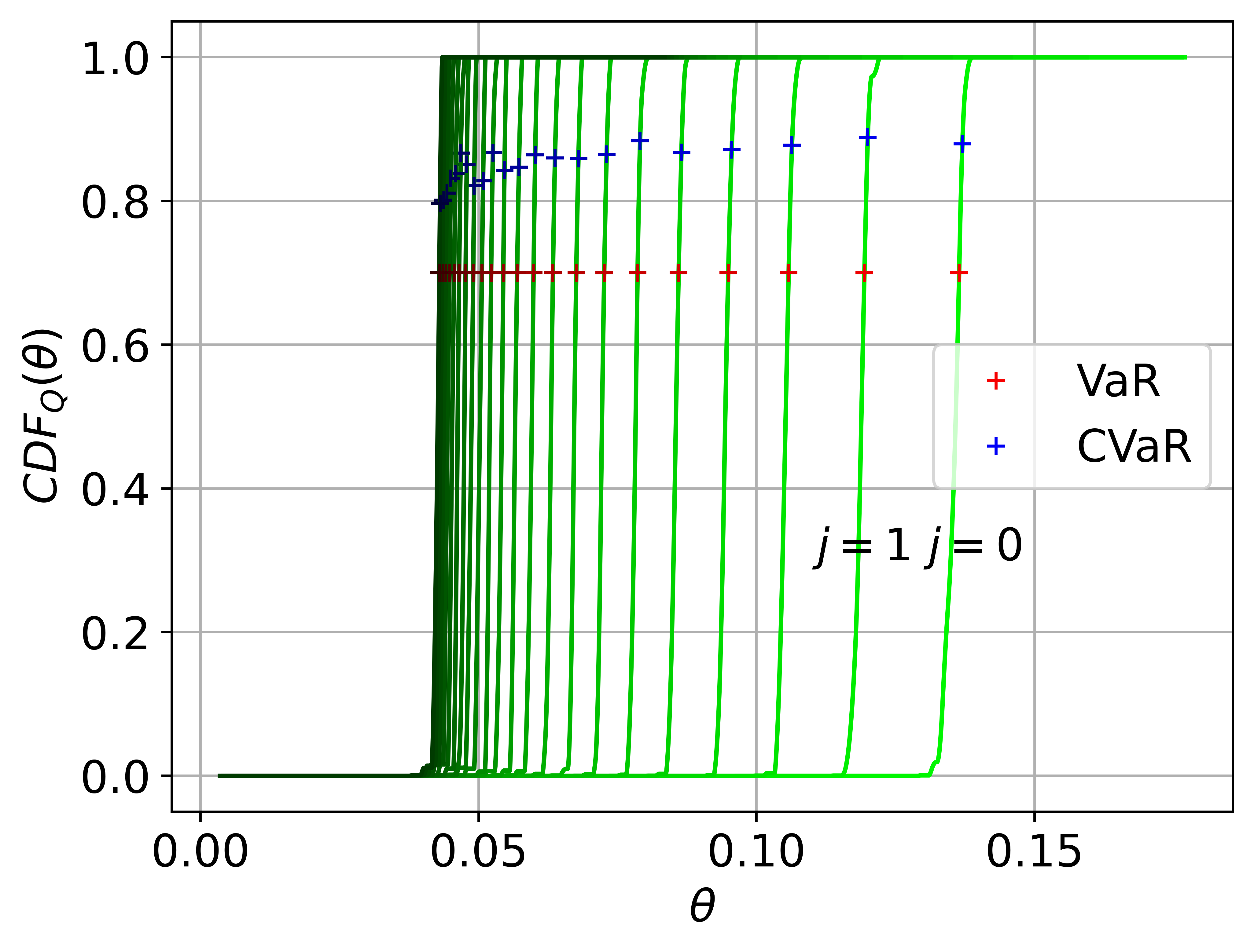}
    \caption{Change in \gls{cdf}.}
    \label{fig:pol_cdf}
  \end{subfigure}
  \caption{Optimization performance over different iterations for the pollutant transport problem}
  \label{fig:pol_opti_performance}
\end{figure}

Fig.~\ref{fig:pol_hierarchy} shows the optimal hierarchy produced by the \gls{cmlmc} algorithm at each iteration of the optimisation for a given tolerance.
Similar to before, we observe that the optimally tuned hierarchy increases in size for later optimisation iterations, since the tolerance supplied to the \gls{cmlmc} is a fraction of the gradient magnitude.
Fig.~\ref{fig:pol_opti_complexity} shows the cumulative cost as defined in Section~\ref{sec:fhn} for a given gradient magnitude.
We observe once again that the cumulative cost grows as $\norm{\objestgrad(w_j)}_{\sltwo}^{-2}$ after an initial pre-asymptotic regime, as is to be expected for the use of an optimally tuned \gls{mlmc} hierarchy at each iteration, tuned to obtain a tolerance proportional to $\norm{\objestgrad(w_j)}_{\sltwo}$.

\begin{figure}[h]
  \centering
  \begin{subfigure}{0.48\textwidth}
    \centering
    \includegraphics[width=\textwidth]{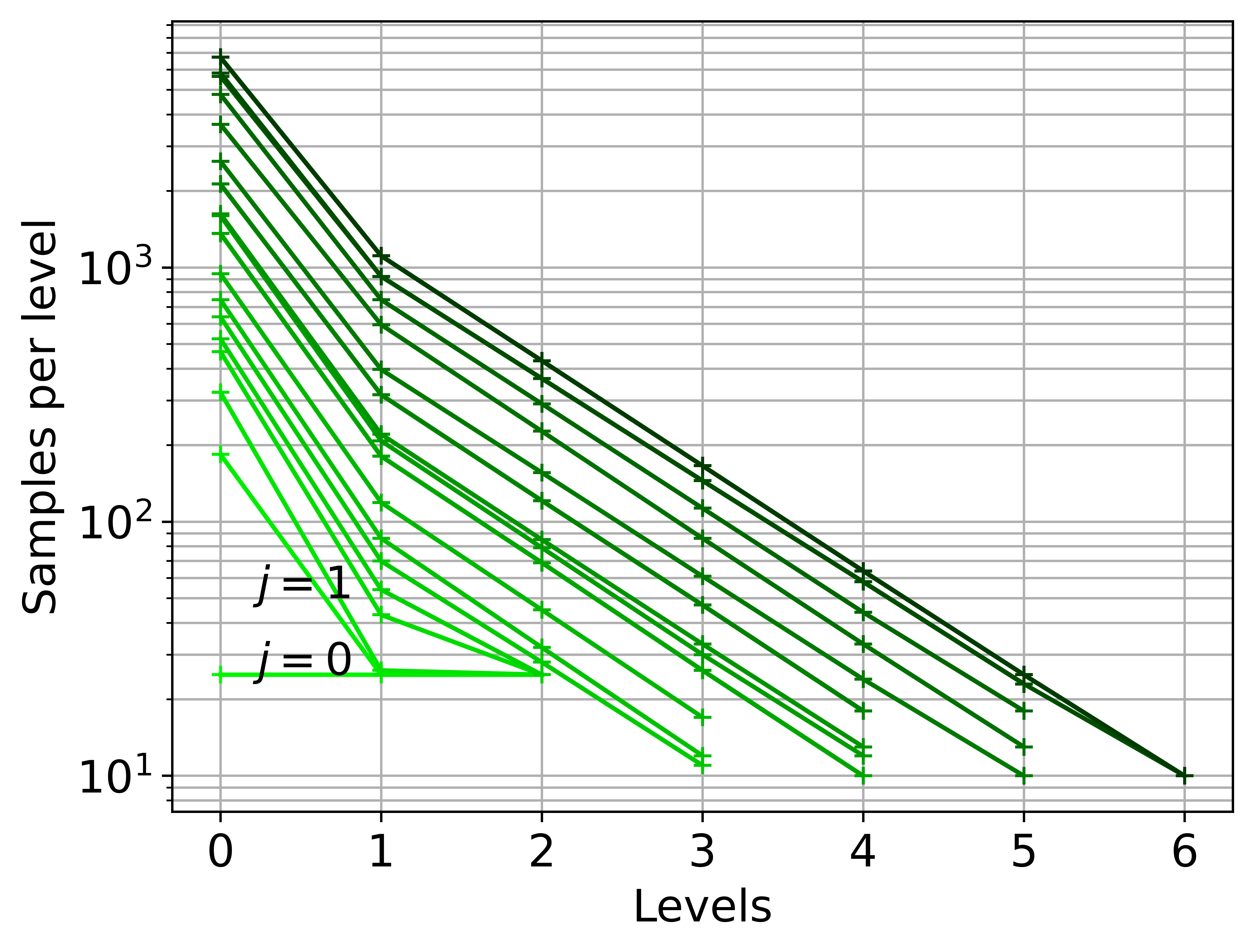}
	\caption{Level-wise sample sizes}
    \label{fig:pol_hierarchy}
  \end{subfigure}
  \begin{subfigure}{0.48\textwidth}
    \centering
    \includegraphics[width=\textwidth]{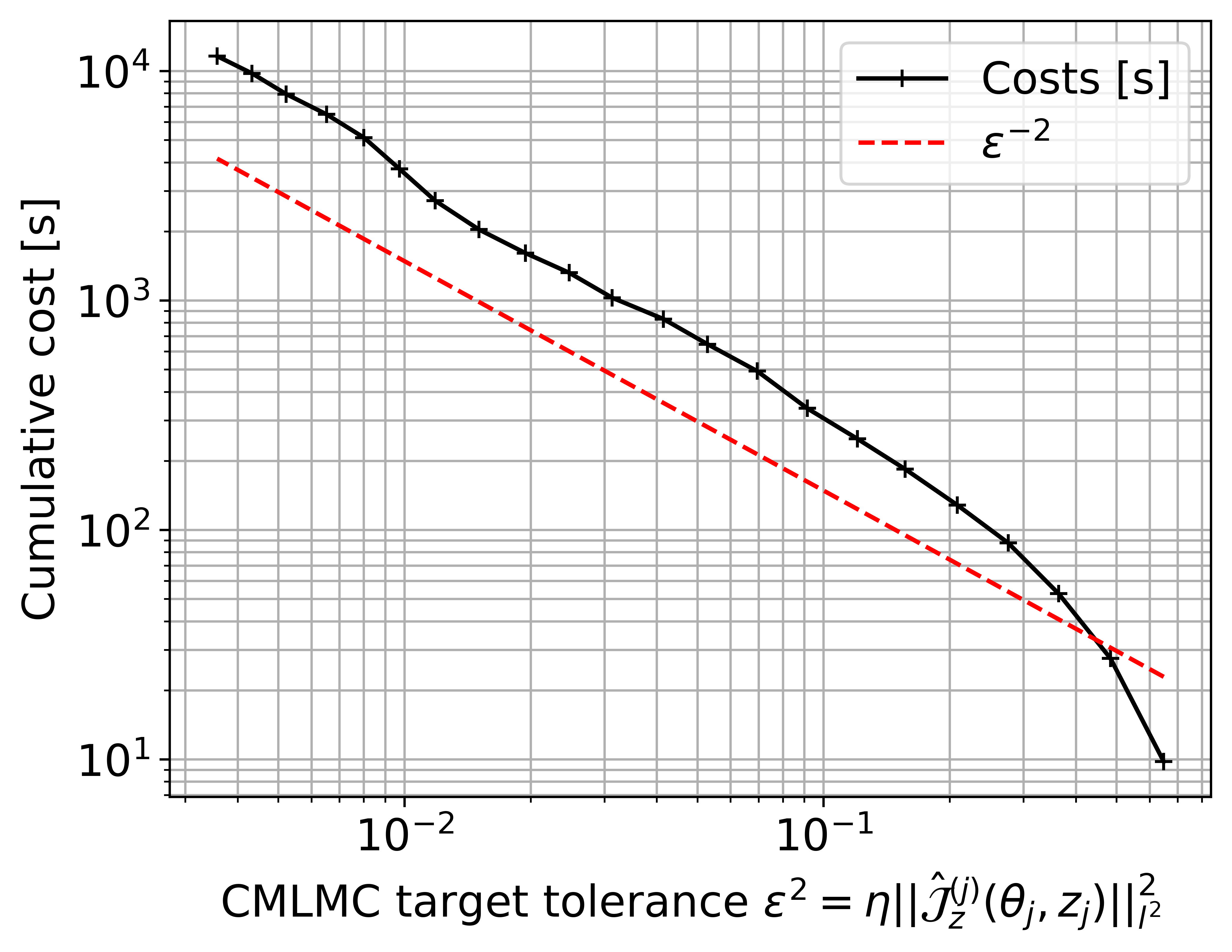}
	\caption{Cumulative cost}
    \label{fig:pol_opti_complexity}
  \end{subfigure}
  \caption{Hierarchy and complexity behaviour for different iterations for the pollutant transport problem}
  \label{fig:pol_comp_performance}
\end{figure}

We now wish to study the performance of the \gls{amgd} algorithm for different significances $\tau$.
To this end, we compare the performance of the algorithm for $\tau=0.7$ and $\tau=0.9$.
Since the performance of the algorithm in terms of objective function and gradient decay in the $\tau=0.9$ case are nearly identical to the performance observed in Fig.~\ref{fig:pol_comp_performance} for the $\tau=0.7$, the corresponding results are not presented here.
Fig.~\ref{fig:pol_hierarchy_sig} shows the optimal hierarchy produced by the \gls{cmlmc} algorithm at each optimisation iteration for the two significances tested.
We observe that the level-wise sample sizes $N_l$ decay at the same rate in the levels $l$ for both tested significances, however with a larger constant for the $\tau=0.9$ case.
Additionally, Fig.~\ref{fig:pol_opti_complexity_sig} shows the cumulative cost for a given gradient magnitude, for both significances.
We observe that the cumulative cost grows as $\norm{\objestgrad(w_j)}_{\sltwo}^{-2}$ in both cases, following an initial pre-asymptotic regime.
However, the $\tau=0.9$ case shows a larger constant.
In this case, the interval $\Theta$ is changed with each optimisation iteration such that it is centered on the quantile estimate corresponding to the previous optimisation iteration.
It can be shown, for the case of a simple Monte Carlo estimator, that the constant is expected to scale in this case as $(1-\tau)^{-1}$.
The interested reader is referred to \cite{Ganesh2022a} for further discussion.
We note that we observe a similar scaling in this case.

\begin{figure}[h]
  \centering
  \begin{subfigure}{0.48\textwidth}
    \centering
    \includegraphics[width=\textwidth]{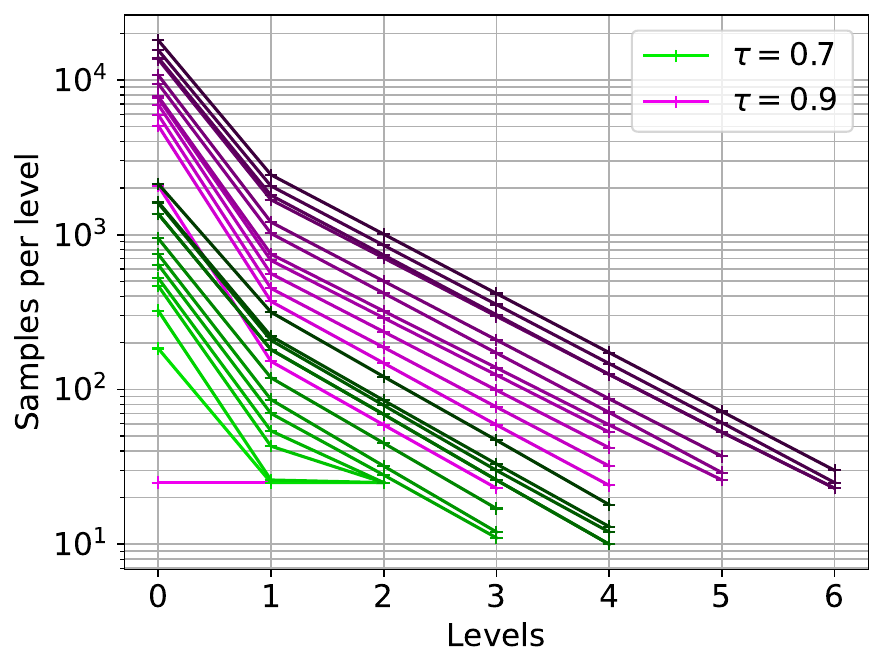}
    \caption{Level-wise sample sizes}
    \label{fig:pol_hierarchy_sig}
  \end{subfigure}
  \begin{subfigure}{0.48\textwidth}
    \centering
    \includegraphics[width=\textwidth]{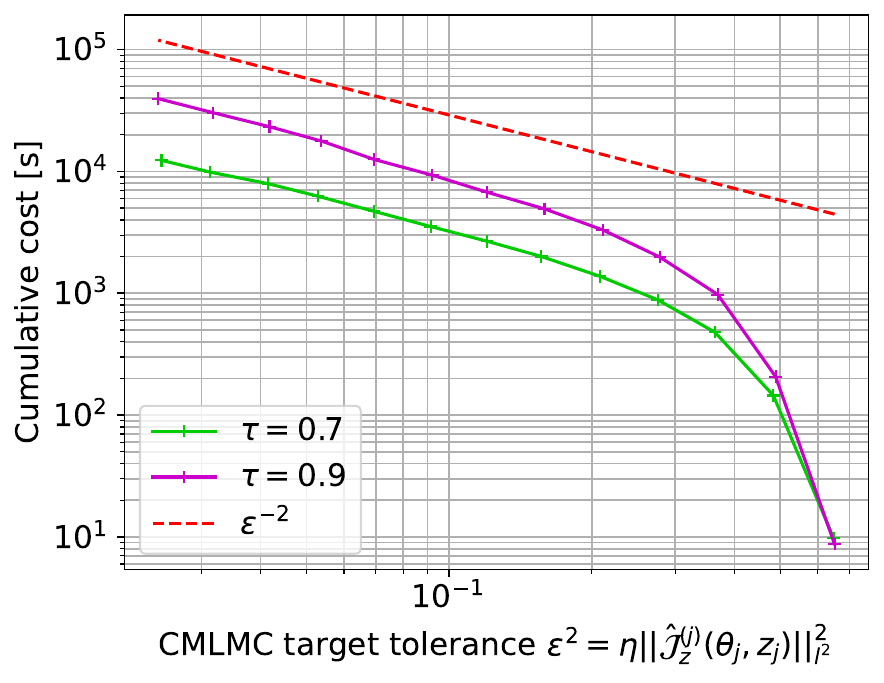}
    \caption{Cumulative cost}
    \label{fig:pol_opti_complexity_sig}
  \end{subfigure}
  \caption{Hierarchy and complexity behaviour for different significances for pollutant transport problem}
  \label{fig:pol_comp_performance_sig}
\end{figure}

\FloatBarrier

\section{Conclusions}
The aim of this work was to tackle the challenge of minimising the \gls{cvar} of a random \gls{qoi}, typically the output of a differential model with random inputs, over a suitable design space, using gradient-based optimisation techniques. 
A main challenge in utilising gradient-based techniques was the differentiability of the \gls{cvar} in terms of the design variables.
A differentiability result was presented in Section~\ref{sec:prob_form}, which was a generalisation of the one presented in \cite{hong2009simulating}, showing that gradient-based algorithms could still be used to directly minimise the \gls{cvar} without requiring smoothing.

The expression for the sensitivities of the \gls{cvar} with respect to design parameters required the computation of expectations of discontinuous functions of the \gls{qoi}; namely, the indicator function.
Estimating this expectation naively using \gls{mlmc} estimators could become impractically expensive, and possibly result in non-optimal complexity behaviour of the corresponding \gls{mlmc} estimator.
A similar issue was discussed and tackled in \cite{Ganesh2022a}, and an alternative was proposed using the framework of parametric expectations.
We presented a modified expression for the sensitivities of the \gls{cvar}, based on derivatives of parametric expectations, thereby allowing us to use the work in \cite{Ganesh2022a}.
Based on this modification, we also presented a novel optimisation algorithm consisting of an alternating minimisation-gradient procedure.
We demonstrated a theoretical result that, under additional assumptions on the combined objective function in Eq.~\eqref{eq:opt_form_combined}, the novel algorithm would achieve Q-linear convergence of the design iterates towards the optimal design in the optimisation iterations.

To enable the use of the work in \cite{Ganesh2022a}, we presented modifications of the \gls{mlmc} estimator, the error estimation procedure and adaptive hierarchy selection procedure specific to computing the sensitivities of the \gls{cvar}.
Namely, a relation was derived between the \gls{mse} of the sensitivities and the \gls{mse} of the parametric expectations in Section~\ref{sec:grad_est_mlmc}.
In addition, a modification of the \gls{kde} smoothing procedure presented in \cite{Ganesh2022a} was presented, specific to \gls{cvar} minimisation. 
The combination of the \gls{mse} relation and \gls{kde} modification allowed us to trivially extend the error estimation and hierarchy adaptivity procedure of \cite{Ganesh2022a} to the current application. 
Lastly, a minor modification of the \gls{cmlmc} procedure of \cite{Ganesh2022a} was presented in Algorithm~\ref{alg:opti_cmlmc}, wherein the \gls{cmlmc} was restarted from the optimal hierarchy of the previous design iterate.

The combination of gradient-based optimisation and \gls{mlmc} estimation of the sensitivities of the \gls{cvar} was tested on two problems of practical relevance; namely the FitzHugh--Nagumo oscillator and a more applied problem of advection-reaction-diffusion problem used to model pollutant transport.
In both cases, it was observed that the novel error estimation procedure provided tight bounds on the \gls{mse} of the gradient as defined in Eq.~\eqref{eq:grad_error_def}.
In addition, the \gls{cmlmc} algorithm was shown to produce the best-case complexity behaviour for the \gls{mlmc} estimators of the sensitivities.
The \gls{ouu} algorithm was shown to converge Q-linearly in the optimisation iterations, while also preserving the best case \gls{mlmc} cost complexity.

The numerical examples considered in this work demonstrated that the \gls{amgd} procedure performs well for the cases presented here.
However, one may wish to improve on the performance of the algorithm by considering alternatives to the \gls{amgd} algorithm.
Such variations could, for example, include higher-order optimisation methods such as the Newton method.
It still remains to be seen whether higher-order methods can directly be used with objective functions of the type in problem~\eqref{eq:opt_form_combined}, as well as whether the framework of parametric expectations can be combined with such an algorithm. 
We plan to explore such questions in future works. 

 \section*{Acknowledgments}
 This project has received funding from the European Union's Horizon
 2020 research and innovation programme under grant agreement
 No. 800898 and the King Abdullah University of Science and Technology
 (KAUST) Office of Sponsored Research (OSR) under Award
 No. OSR-2019-CRG8-4033.

\begin{appendices}
\section{Proof of Theorem~\ref{thm:q_t_diff}} \label{sec:proof-q-t-diff}
To prove Theorem~\ref{thm:q_t_diff} on the Fr\'echet differentiability of the objective function $\obj(\theta,z)$, we first prove an important result in Lemma~\ref{lemma:dom_conv}.
We recall that $\Gamma \subset \lp(\Omega,\setR)$ is the set of $\lp$-integrable random variables whose measures are atom-free.
\begin{lemma}\label{lemma:dom_conv}
Consider random variables $Y \in \Gamma \subset \lp(\Omega, \setR)$ and $\delta Y \in \lp(\Omega, \setR)$.
We then have the following:
\begin{align}
\lim_{\norm{\delta Y}_{\lp} \to 0} \expec{\indicator_{\{0 \leq Y \leq -\delta Y\}}} &= 0, \label{eq:lem_pos}\\
\text{and}\quad \lim_{\norm{\delta Y}_{lp} \to 0} \expec{\indicator_{\{-\delta Y \leq Y \leq 0\}}} &= 0.\label{eq:lem_neg}
\end{align}
\begin{proof}
We begin with the proof for Eq.~\eqref{eq:lem_pos}, since the proof for Eq.~\eqref{eq:lem_neg} follows from identical arguments.
We make use of the following result; for any $X \in \lp(\Omega, \setR)$, the following holds for any $\epsilon > 0$ and $p \geq 0$:
\begin{align}
\expec{\indicator_{\{\abs{X} \geq \epsilon\}}} \leq \expec{\frac{\abs{X}^p}{\epsilon^p}} = \norm{X}^p_{\lp} \epsilon^{-p}.
\end{align}
Setting $\epsilon = \norm{X}_{\lp}^{\beta}$ for some $\beta \in [0,1)$, we have that 
\begin{align}
\expec{\indicator_{\{\abs{X} \geq \norm{X}_{\lp}^{\beta}\}}} \leq \norm{X}^{p-\beta p}_{\lp} = \norm{X}^{\gamma}_{\lp},
\end{align}
where $\gamma \coloneqq p (1-\beta )$.
We rewrite the term within the limit in Eq.~\eqref{eq:lem_pos} as follows:
\begin{align}
\expec{\indicator_{\{0 \leq Y \leq -\delta Y\}}} &= \expec{\indicator_{\{0 \leq Y \leq -\delta Y\}} \left(\indicator_{\{\abs{\delta Y} < \norm{\delta Y}_{\lp}^{\beta}\}} + \indicator_{\{\abs{\delta Y} \geq \norm{\delta Y}_{\lp}^{\beta}\}} \right)} \\
&= \expec{\indicator_{\{0 \leq Y \leq -\delta Y\}} \indicator_{\{\abs{\delta Y} < \norm{\delta Y}_{\lp}^{\beta}\}} }+\expec{\indicator_{\{0 \leq Y \leq -\delta Y\}} \indicator_{\{\abs{\delta Y} \geq \norm{\delta Y}_{\lp}^{\beta}\}} }.\label{eq:two_terms}
\end{align}
The first term can be bounded as follows:
\begin{align}
\expec{\indicator_{\{0 \leq Y \leq -\delta Y\}} \indicator_{\{\abs{\delta Y} < \norm{\delta Y}_{\lp}^{\beta}\}} } & \leq \expec{\indicator_{\{0 \leq Y \leq \norm{\delta Y}_{\lp}^{\beta}\}} }
\end{align}
Due to dominated convergence, we can pass the limit into the expectation, resulting in the following:
\begin{align}
\lim_{\norm{\delta Y}_{\lp} \to 0} \expec{\indicator_{\{0 \leq Y \leq \norm{\delta Y}_{\lp}^{\beta}\}} } &= \expec{\lim_{\norm{\delta Y}_{\lp} \to 0}\indicator_{\{0 \leq Y \leq \norm{\delta Y}_{\lp}^{\beta}\}}} = \expec{\indicator_{\{ Y = 0 \}}} =0,
\end{align}
since $Y \in \Gamma$ is atom-free. 
The second term can be bounded as follows, where we use a H\"older inequality:
\begin{align}
\expec{\indicator_{\{0 \leq Y \leq -\delta Y\}} \indicator_{\{\abs{\delta Y} \geq \norm{\delta Y}_{\lp}^{\beta}\}} } &\leq \norm{\indicator_{\{0 \leq Y \leq -\delta Y\}}}_{\linf} \norm{\indicator_{\{\abs{\delta Y} \geq \norm{\delta Y}_{\lp}^{\beta}\}}}_{\lebesgue^1} \\
&\leq \expec{\indicator_{\{\abs{\delta Y} \geq \norm{\delta Y}_{\lp}^{\beta}\}}} \\
&\leq \norm{\delta Y}^{\gamma}_{\lp}.
\end{align}
Hence, we have that the second term in Eq.~\eqref{eq:two_terms} goes to zero as well with the application of the limit, thus concluding the proof for Eq.~\eqref{eq:lem_pos}. 
The proof for Eq.~\eqref{eq:lem_neg} follows from identical arguments.
\end{proof}
\end{lemma}
We now use the above result and present a proof of Theorem~\ref{thm:q_t_diff}.
We note that the function $\risk(\theta, \qoi)$ is a composition of two functions.
We define the functions $l_1 : \Gamma \to \setR$ and $l_2 : \setR \times \Gamma \to \Gamma$ as follows:
\begin{align}
l_1(Y) &\coloneq \expec{Y^+}\\
l_2(\theta, \qoi) &\coloneqq \qoi - \theta,\\
\implies \risk(\theta, \qoi) &= \theta + \frac{l_1 \circ l_2 (\theta, \qoi)}{1-\tau}.
\end{align}
Hence, to show that $\risk$ is Fr\'echet differentiable, it suffices to show that each of the functions $l_1$ and $l_2$ are Fr\'echet differentiable. 

It is straightforward to see that $l_2$ is Fr\'echet differentiable (being linear and bounded) with Fr\'echet derivative $Dl_2(\theta, \qoi)$ in the direction $(\delta \theta, \delta \qoi) \in \setR \times \lp(\Omega, \setR)$ given by:
\begin{align}
Dl_2(\theta, \qoi)(\delta \theta, \delta \qoi) = \delta \qoi - \delta \theta.
\end{align}
The Fr\'echet derivative of $l_1$ however, requires some consideration. 
We argue that the Fr\'echet derivative of $l_1$ exists at any point $Y \in \Gamma$ and is given by $Dl_1(Y)(\delta Y) = \expec{\indicator_{\{Y \geq 0\}} \delta Y}$.
To prove this statement, we must verify the following limit:
\begin{align}
\lim_{\norm{\delta Y}_{\lp}\to 0 } \frac{\abs{\expec{(Y+\delta Y)^+} - \expec{Y^+}-\expec{\indicator_{\{Y \geq 0\}}\delta Y}}}{\norm{\delta Y}_{\lp}} = 0 \label{eq:limit}
\end{align}
To show the above, we begin by re-writing the numerator as follows:
\begin{align}
\expec{(Y+\delta Y)^+ - Y^+ - \indicator_{\{Y \geq 0\}}\delta Y} &= \expec{\delta Y\indicator_{\{Y+\delta Y \geq 0, Y \geq 0\}}-\indicator_{\{Y \geq 0\}}\delta Y}\nonumber\\
&+ \expec{(Y+\delta Y)\indicator_{\{Y+\delta Y \geq 0, Y < 0\}}}\nonumber\\
&- \expec{Y \indicator_{\{Y+\delta Y < 0, Y \geq 0\}}}.\label{eq:three_split}
\end{align}
Inserting Eq.~\eqref{eq:three_split} into Eq.~\eqref{eq:limit}, we have the following:
\begin{align}
\frac{\abs{\expec{(Y+\delta Y)^+} - \expec{Y^+}-\expec{\indicator_{\{Y \geq 0}\delta Y\}}}}{\norm{\delta Y}_{\lp}} \leq \frac{T_1 + T_2 + T_3}{\norm{\delta Y}_{\lp}},\label{eq:three_bound}
\end{align}
with the terms $T_1$, $T_2$ and $T_3$ given by:
\begin{align}
T_1 &\coloneqq \abs{\expec{\delta Y\indicator_{\{Y+\delta Y \geq 0, Y \geq 0\}}-\delta Y \indicator_{\{Y \geq 0\}}}},\\
T_2 &\coloneqq \abs{\expec{(Y+\delta Y)\indicator_{\{Y+\delta Y \geq 0, Y < 0\}}} },\\
T_3 &\coloneqq \abs{ \expec{Y\indicator_{\{Y+\delta Y < 0, Y \geq 0\}}}}.
\end{align}
We then begin with the term $T_1$.
We first note that $T_1$ can be rewritten in the following manner:
\begin{align}
T_1 &= \abs{- \expec{\delta Y \indicator_{\{0 \leq Y < -\delta Y\} }}} \leq \expec{\abs{\delta Y}\indicator_{\{0 \leq Y < -\delta Y \}}}\\
&\leq \norm{\delta Y}_{\lp} \norm{\indicator_{\{0 \leq Y < -\delta Y \}}}_{\ellq} = \norm{\delta Y}_{\lp} \expec{\indicator_{\{0 \leq Y < -\delta Y \}}}^{1/q},\\
&\leq \norm{\delta Y}_{\lp} \expec{\indicator_{\{0 \leq Y \leq -\delta Y \}}}^{1/q}.\label{eq:t1_bound}
\end{align}
The term $T_2$ can be bounded as follows:
\begin{align}
T_2 &= \abs{\expec{(Y+\delta Y) \indicator_{\{Y+\delta Y \geq 0\}} \indicator_{\{Y < 0\}}}} \leq \expec{|\delta Y| \indicator_{\{-\delta Y \leq Y < 0\}} },\\
&\leq \norm{\delta Y}_{\lp} \norm{\indicator_{\{-\delta Y \leq Y < 0\}}}_{\ellq} = \norm{\delta Y}_{\lp} \expec{\indicator_{\{-\delta Y \leq Y \leq 0\}}}^{1/q}. \label{eq:t2_bound}
\end{align}
Similarly, the term $T_3$ can be bounded as follows:
\begin{align}
T_3 &= \abs{\expec{Y \indicator_{\{Y+\delta Y < 0\}} \indicator_{\{Y \geq 0\}}} } \leq  \expec{ \abs{\delta Y} \indicator_{\{0 \leq Y< -\delta Y\}} }\\
&\leq \norm{\delta Y}_{\lp} \norm{\indicator_{\{0 \leq Y \leq -\delta Y\}} }_{\ellq}= \norm{\delta Y}_{\lp} \expec{\indicator_{\{0 \leq Y \leq -\delta Y\}}}^{1/q}.\label{eq:t3_bound}
\end{align}
Inserting Eqs.~\eqref{eq:t1_bound}, \eqref{eq:t2_bound} and~\eqref{eq:t3_bound} into Eq.~\eqref{eq:three_bound}, and applying the limit using Lemma~\ref{lemma:dom_conv}, we have that:
\begin{align}
\lim_{\norm{\delta Y}_{\lp}\to 0 } \frac{\abs{\expec{(Y+\delta Y)^+} - \expec{Y^+}-\expec{\indicator_{\{Y \geq 0\}}\delta Y}}}{\norm{\delta Y}_{\lp}} = 0.
\end{align}
This concludes the proof.

\section{Adjoint of first-order \texorpdfstring{\glsname{ode}}{ODE} with additive noise}
\label{sec:proof-oscillator}
We present here the derivation of the adjoints for a first-order \gls{ode} with white noise forcing for an objective function containing the \gls{cvar} of a time-averaged quantity of the trajectory. 
Let $(\Omega, \mathcal{F}, \measure)$ be a complete probability space, $\omega \in \Omega$ denote an elementary random event, and $z \in \setR^d$ the set of design variables.
Let $u(t, z, \omega) \in U \subset \setR^{N_u}$ be the state vector at time $t \in[0,T]$ for a given random input $\omega$ and design $z$. The state vector $u$ is governed by the following \gls{ode} with additive noise.
\begin{align}
  \dot{u}(t,z,\omega) &= g(u,z) + \tau \dot{W}(t,\omega) \quad\text{over } (0,T],\\
  u(0, z,\omega) &= u^0,
\end{align}
where $g : U \times \setR^d \to \setR^{N_u}$, and \(W:[0,T] \times \Omega \to\setR^{N_u}\) is a $N_u$-dimensional standard Wiener process.

We discretise the problem on a uniform temporal grid \(\mathbb{T}\) where the interval \([0,T]\) is divided into \(N\in\setN\) segments of step size $\Delta t = T/N$, \(\mathbb{T}\coloneqq \{t_{n}\coloneqq n\Delta t : n\in\Zint{0}{N_{l}}\}\).
The \gls{ode} is discretised using the Euler--Maruyama scheme, which reads as follows:
\begin{align*}
  u^{n+1} &= u^{n} + \Delta t g(u^n, z) + \tau \sqrt{\Delta t} \xi^n,\\
  u^0 &= u_0,
\end{align*}
where $u^n$ denotes the approximation to $u(t_n, z,\omega)$, $\xi^n \in \setR^{N_u}$ are $N_u$-dimensional random vectors whose components are independent identically distributed standard normal variables. 
We are interested in computing the statistics of time-averages of functions of the trajectory. 
\begin{align}
  \qoi = \tavg[T]{f(u)}.
\end{align}
We approximate the time integral using the trapezoid rule on the aforementioned temporal grid, leading to 
\begin{align}
  \qoi(z,\omega) \approx \qoi_h(z,\omega) \coloneqq \sum_{n=0}^{N-1} \left( \frac{f(u^n) + f(u^{n+1})}{2} \right) \frac{\Delta t}{T}.
\end{align}
We are interested in minimising the \gls{cvar} of this quantity over the parameters $z$ but use the combined formulation in Eq.~\eqref{eq:opt_form_combined}.
The corresponding Lagrangian for the problem reads 
\begin{align}
  \lagrangian(\theta,z,\{u^n\},\{\lambda^n\}) = \theta + \frac{\expec{(\qoi(z,\cdot)-\theta)^{+}}}{1-\tau} + \expec{\sum_{n=0}^{N-1}\lambda^{n+1}\left(u^{n} + \Delta t g^n + \tau \sqrt{\Delta t} \xi^n - u^{n+1} \right) - \lambda^0(u^0-u_0)},
\end{align}
where we use $g^n \coloneqq g(u^n,z)$, and $\lambda^{n} \in \setR^{N_u}$, $n\in\Zint{0}{N}$ denote the Lagrange multipliers for the initial condition and the steps of the discretised equations. 

Differentiating with respect to $z$ gives
\begin{align}
  \deriv{\lagrangian}{z} &= \begin{multlined}[t]
    \expec{ \frac{\indicator_{\qoi_h\geq \theta}}{(1-\tau)T} \sum_{n=0}^{N-1} \left(\frac{f^n_u u^n_z + f^{n+1}_u u^{n+1}_z}{2}\right) \Delta t } \\
    + \expec{\sum_{n=0}^{N-1}\lambda^{n+1}\left(u_z^{n} + \Delta t (g_u^n u_z^n + g_z^n) - u_z^{n+1} \right)}
  \end{multlined} \\
  & \eqqcolon \expec{\hat{\lagrangian}}.
\end{align}
Re-arranging the terms leads to 
\begin{multline}
  \hat{\lagrangian} = u_z^0 \left[ \lambda^1 (1+\Delta t g_u^0) +\frac{\indicator_{\qoi_h\geq \theta}}{(1-\tau)T} \frac{f_u^0 \Delta t}{2} \right] + \Delta t \lambda^1 g_z^0 + u_z^N \left[\frac{\indicator_{\qoi_h\geq \theta}}{(1-\tau)T}\frac{f_u^N \Delta t}{2}-\lambda^N\right]\\
  + \sum_{n=1}^{N-1} u_z^n \left[ \lambda^{n+1}(1+\Delta t g_u^n)-\lambda^n + \frac{\indicator_{\qoi_h\geq \theta}}{(1-\tau)T} \Delta t f_u^n \right] + \Delta t \lambda^{n+1} g_z^n,
\end{multline}
where we have used the subscript notation for partial derivatives.

We have in our case that $u_z^0 = 0$. 
To remove terms dependent on $u_z^n$, we set 
\begin{align}
  \lambda^n &= \lambda^{n+1}(1+\Delta t g_u^n)+ \frac{\indicator_{\qoi_h\geq \theta}}{(1-\tau)T} \Delta t f_u^n, \quad n=1,...,N-1 \label{eq:lambda_time}\\
  \lambda^N &= \frac{\indicator_{\qoi_h\geq \theta}}{(1-\tau)T}\frac{f_u^N \Delta t}{2}.
\end{align}
This gives us the adjoint equations which are solved backwards in time. 
It is noteworthy to mention that since Eq.~\eqref{eq:lambda_time} is linear, that it can be solved for $\{\lambda^n\}$ without the factor $\frac{\indicator_{\qoi_h \geq \theta}}{(1-\tau)T}$, and equivalently, the sensitivities can be computed as:
\begin{align}
  \deriv{\lagrangian}{z} = \frac{\indicator_{\qoi_l \geq \theta}}{(1-\tau)T} \expec{ \sum_{n=0}^{N-1} \Delta t \lambda^{n+1} g_z^n}.
\end{align}
That is, setting 
\begin{align}
  \obj(\theta,z) = \theta + \frac{\expec{(\qoi_h(z,\cdot)-\theta)^{+}}}{1-\tau},
\end{align}
we have that $\obj_z(\theta,z) = \expec{ \sum_{n=0}^{N-1} \Delta t \lambda^{n+1} g_z^n}$.
\end{appendices}

\printbibliography

\end{document}